\documentclass{amsart}
\usepackage{amsmath, amssymb, amsthm, epsfig}
\usepackage{hyperref, latexsym}
\usepackage{url} 
\usepackage{cellspace}
\DeclareMathOperator{\sgn}{\mathrm{sgn}}

\DeclareMathOperator{\Vol}{\mathrm{Vol}}               

\begin{document}
 \bibliographystyle{plain}

 \newtheorem{theorem}{Theorem}[section]
 \newtheorem{lemma}[theorem]{Lemma}
 \newtheorem{proposition}[theorem]{Proposition}
 \newtheorem{corollary}[theorem]{Corollary}
 \theoremstyle{definition}
 \newtheorem{definition}[theorem]{Definition}
 \newtheorem{example}[theorem]{Example}
 \newcommand{\mc}{\mathcal}
 \newcommand{\A}{\mc{A}}
 \newcommand{\B}{\mc{B}}
 \newcommand{\cc}{\mc{C}}
 \newcommand{\D}{\mc{D}}
 \newcommand{\E}{\mc{E}}
 \newcommand{\F}{\mc{F}}
 \newcommand{\G}{\mc{G}}
 \newcommand{\sH}{\mc{H}}
 \newcommand{\I}{\mc{I}}
 \newcommand{\J}{\mc{J}}
 \newcommand{\K}{\mc{K}}
 \newcommand{\lL}{\mc{L}}
 \newcommand{\M}{\mc{M}}
 \newcommand{\nn}{\mc{N}}
 \newcommand{\rr}{\mc{R}}
 \newcommand{\sS}{\mc{S}}
 \newcommand{\U}{\mc{U}}
 \newcommand{\X}{\mc{X}}
 \newcommand{\Y}{\mc{Y}}
 \newcommand{\zz}{\mc{Z}}
 \newcommand{\C}{\mathbb{C}}
 \newcommand{\R}{\mathbb{R}}
 \newcommand{\N}{\mathbb{N}}
 \newcommand{\Q}{\mathbb{Q}}
 \newcommand{\T}{\mathbb{T}}
 \newcommand{\Z}{\mathbb{Z}}
 \newcommand{\csch}{\mathrm{csch}}
 \newcommand{\tF}{\widehat{F}}
 \newcommand{\tG}{\widehat{G}}
 \newcommand{\tH}{\widehat{H}}
 \newcommand{\tf}{\widehat{f}}
 \newcommand{\ug}{\widehat{g}}
 \newcommand{\wg}{\widetilde{g}}
 \newcommand{\uh}{\widehat{h}}
 \newcommand{\wh}{\widetilde{h}}
 \newcommand{\tJ}{\widehat{J}}
 \newcommand{\tk}{\widehat{k}}
 \newcommand{\tK}{\widehat{K}}
 \newcommand{\tl}{\widehat{l}}
 \newcommand{\tL}{\widehat{L}}
 \newcommand{\tm}{\widehat{m}}
 \newcommand{\tM}{\widehat{M}}
 \newcommand{\tp}{\widehat{\varphi}}
 \newcommand{\tq}{\widehat{q}}
 \newcommand{\ts}{\widehat{s}}
 \newcommand{\tS}{\widehat{S}}
 \newcommand{\tsigma}{\widehat{\sigma}}
 \newcommand{\ttau}{\widehat{\tau}}
 \newcommand{\ttt}{\widehat{t}}
 \newcommand{\tT}{\widehat{T}}
 \newcommand{\tU}{\widehat{U}}
 \newcommand{\tu}{\widehat{u}}
 \newcommand{\tV}{\widehat{V}}
 \newcommand{\tv}{\widehat{v}}
 \newcommand{\tW}{\widehat{W}}
 \newcommand{\ba}{\boldsymbol{a}}
 \newcommand{\bb}{\boldsymbol{b}}
 \newcommand{\bal}{\boldsymbol{\alpha}}
 \newcommand{\bbeta}{\boldsymbol{\beta}}
 \newcommand{\bdelta}{\boldsymbol{\delta}}
 \newcommand{\bpsi}{\boldsymbol{\psi}}
 \newcommand{\bd}{\boldsymbol{d}}
 \newcommand{\be}{\boldsymbol{e}}
 \newcommand{\bl}{\boldsymbol{l}}
 \newcommand{\bk}{\boldsymbol{k}}
 \newcommand{\bell}{\boldsymbol{\ell}}
 \newcommand{\bL}{\boldsymbol{L}}
 \newcommand{\bm}{\boldsymbol{m}}
 \newcommand{\bn}{\boldsymbol{n}}
 \newcommand{\bu}{\boldsymbol{u}}
 \newcommand{\bv}{\boldsymbol{v}}
 \newcommand{\bw}{\boldsymbol{w}}
 \newcommand{\bx}{\boldsymbol{x}}
 \newcommand{\bxi}{\boldsymbol{\xi}}
 \newcommand{\ceta}{\boldsymbol{\eta}}
 \newcommand{\blambda}{\boldsymbol{\lambda}}
 \newcommand{\bwy}{\boldsymbol{y}}
 \newcommand{\bzero}{\boldsymbol{0}}
 \newcommand{\bone}{\boldsymbol{1}}
 \newcommand{\ep}{\varepsilon}
 \newcommand{\p}{\varphi}
 \newcommand{\f}{\frac52}
 \newcommand{\g}{\frac32}
 \newcommand{\h}{\frac12}
 \newcommand{\hh}{\tfrac12}
 \newcommand{\ds}{\text{\rm d}s}
 \newcommand{\dt}{\text{\rm d}t}
 \newcommand{\du}{\text{\rm d}u}
 \newcommand{\dv}{\text{\rm d}v}
 \newcommand{\dbw}{\text{\rm d}\bw}
 \newcommand{\dx}{\text{\rm d}x}
 \newcommand{\dbx}{\text{\rm d}\bx}
 \newcommand{\dy}{\text{\rm d}y}
 \newcommand{\dbwy}{\text{\rm d}\bwy}
 \newcommand{\dl}{\text{\rm d}\lambda}
 \newcommand{\dmu}{\text{\rm d}\mu}
 \newcommand{\dnu}{\text{\rm d}\nu(\lambda)}
 \newcommand{\dnus}{\text{\rm d}\nu_{\sigma}(\lambda)}
 \newcommand{\dlnu}{\text{\rm d}\nu_l(\lambda)}
 \newcommand{\dnnu}{\text{\rm d}\nu_n(\lambda)}
 \newcommand{\sech}{\text{\rm sech}}
 \def\today{\number\time, \ifcase\month\or
  January\or February\or March\or April\or May\or June\or
  July\or August\or September\or October\or November\or December\fi
  \space\number\day, \number\year}

\title[Fractional parts]{Sums of products of fractional parts}
\author[L\^e and Vaaler]{Th\'ai~Ho\`ang~L\^e and Jeffrey~D.~Vaaler}
\date{\today}
\keywords{linear forms, fractional part, Littlewood conjecture}
\thanks{The research of the second author was supported by NSA grant, H98230-12-1-0254.}

\address{Department of Mathematics, Univerisity of Texas, Austin, Texas 78712-1082 USA}
\email{leth@math.utexas.edu}

\address{Department of Mathematics, University of Texas, Austin, Texas 78712-1082 USA}
\email{vaaler@math.utexas.edu}

\begin{abstract} We prove upper and lower bounds for certain sums of products of fractional parts by
using majoring and minorizing functions from Fourier analysis.  In special cases
the upper bounds are sharp if there exist counterexamples to the Littlewood conjecture in Diophantine
approximation.  We introduce a generalization of such counterexamples which we call strongly badly
approximable matrices.  We also prove a transference principle for strongly badly approximable matrices.
\end{abstract}

\maketitle

\numberwithin{equation}{section}

\section{Introduction}

In this paper we prove upper and lower bounds for certain sums of products of fractional parts.  As usual we write
\begin{equation}\label{intro1}
\|x\| = \min\{|x - n|: n\in\Z\}
\end{equation}
for the distance from the real number $x$ to the nearest integer.  Then $x\mapsto \|x\|$ is well defined on the quotient
group $\R/\Z$, and $(x, y)\mapsto \|x - y\|$ is a metric on $\R/\Z$ that induces its quotient topology.  

We will work in the following general setting.  Let $M$ and $N$ be positive integers, and then define compact abelian groups
\begin{equation}\label{intro2}
G_1 = (\R/\Z)^{MN},\quad\text{and}\quad G_2 = (\R/\Z)^M.
\end{equation}
We write $\mu_1$ and $\mu_2$, respectively, for Haar measures on the Borel subsets of these groups normalized so
that $\mu_1(G_1) = \mu_2(G_2) = 1$.  We write the elements of the group $G_1$ as $M\times N$ matrices with entries in
$\R/\Z$.  That is, we write
\begin{equation*}\label{intro3}
A = (\alpha_{mn}),\quad\text{where}\quad \alpha_{mn} \in \R/\Z, 
\end{equation*}
for a generic element of $G_1$.  Obviously addition in the group $G_1$ coincides with addition of matrices.
We write the elements of the group $G_2$ as $M \times 1$ column matrices.  If $\bxi$ is an $N \times 1$ column vector in
$\Z^N$ then 
\begin{equation}\label{intro4}
A \mapsto A\bxi = \biggl(\sum_{n=1}^N \alpha_{mn} \xi_n\biggr)
\end{equation}
defines a continuous homomorphism from $G_1$ into $G_2$.  If $\bxi \not= \bzero$ then it follows easily that (\ref{intro4})
is surjective and measure preserving.

For positive integers $L_1, L_2, \dots , L_M$, we define $F_{\bL}: G_2 \rightarrow [1, \infty)$ by
\begin{equation*}\label{intro5}
F_{\bL}(\bx) = \prod_{m=1}^M \min\big\{L_m, (2 \|x_m\|)^{-1}\big\}.
\end{equation*}
And we define $F: G_2 \rightarrow [1, \infty]$ by
\begin{equation}\label{intro6}
F(\bx) = \prod_{m=1}^M (2 \|x_m\|)^{-1}.
\end{equation}
If $A$ is an element of the group $G_1$, if $Y\subseteq \Z^N$ is a finite set of integer lattice points, and $X = Y - Y$ is its difference 
set, we prove lower bounds for the sum
\begin{equation*}\label{intro7}
\sum_{\substack{\bxi \in X\\\bxi \not= \bzero}} F_{\bL}(A\bxi).
\end{equation*}
Our lower bounds depend on the cardinality $|Y|$, and on the integers $L_1, L_2, \dots , L_M$, but {\it not} on the point $A$ in $G_1$.
Here is the precise inequality.

\begin{theorem}\label{thmintro1}  Let $L_1, L_2, \dots , L_M$ be positive integers, $Y \subseteq \Z^N$ a finite, nonempty subset of integer lattice
points with difference set $X = Y - Y$.  Then for every point $A$ in the group $G_1$, we have
\begin{equation}\label{intro10}
|Y| \prod_{m=1}^M \log (L_m + 1) - \prod_{m=1}^M L_m \le \sum_{\substack{\bxi \in X\\\bxi \not= \bzero}} F_{\bL}(A\bxi).
\end{equation}
\end{theorem}

Obviously (\ref{intro10}) is of interest when the left hand side is positive, and so when $|Y|$ is sufficiently large.  In the
very special case $M = N = 1$ and $Y = \{0, 1, 2, \dots , K\}$, the inequality (\ref{intro10}) states that
\begin{equation}\label{intro11}
(K+1)\log(L+1) - L \le \sum_{\substack{k=-K\\k \not= 0}}^K \min\big\{L, (2\|k\alpha\|)^{-1}\big\} 
			\le \sum_{\substack{k=-K\\k \not= 0}}^K (2\|k\alpha\|)^{-1}
\end{equation}
for all $\alpha$ in $\R/\Z$.  By taking $L = K$ in (\ref{intro11})
we obtain a more precise form of an inequality obtained by Hardy and Littlewood in \cite{HL1922} and \cite{HL1930}.
Their method, and later refinements of Haber and Osgood \cite[Theorem 2]{HO}, made assumptions on the continued fraction 
expansion of $\alpha$.  The method we develop here uses a special trigonometric polynomial with nonnegative Fourier
coefficients that minorizes the function $\bx \mapsto F_{\bL}(\bx)$.  Once such a trigonometric polynomial is determined, we appeal to ideas
formulated by H.~L.~Montgomery in \cite[Chapter 5, Theorem 9]{HLM}.  

More generally, we can set $L_1 = L_2 = \cdots = L_M$ in (\ref{intro10}), and then select this common value in an optimal way.
This leads to the following lower bound for sums of the simpler expression $F(A\bxi)$.

\begin{corollary}\label{corintro2}  Let $Y \subseteq \Z^N$ be a finite, nonempty subset of integer lattice
points with difference set $X = Y - Y$.  Then for every point $A$ in the group $G_1$ we have
\begin{equation}\label{intro14}
|Y|\biggl(\frac{\log |Y|}{M}\biggr)^M - |Y| \le \sum_{\substack{\bxi \in X\\\bxi \not= \bzero}} F(A\bxi).
\end{equation}
\end{corollary}

We also prove upper bounds for sums of the form
\begin{equation}\label{intro15}
\sum_{\substack{\bxi \in X\\\bxi \not= \bzero}} F(A\bxi),
\end{equation}
where $X$ is a finite, nonempty subset of lattice points in $\Z^N$, but now we no longer assume that $X$ is a difference set.
If $X$ contains a nonzero point then it is clear that the sum (\ref{intro15}) is not a bounded function of $A$, and therefore 
no uniform upper bound for all $A$ in $G_1$ is possible.  However, we are able to give an upper bound, comparable with 
the lower bound (\ref{intro14}), that holds at all points $A$ in $G_1$ outside a subset of small $\mu_1$-measure.

\begin{theorem}\label{thmintro3}   Let $X \subseteq \Z^N$ be a finite, nonempty subset of lattice points with cardinality $|X|$.  If 
$0 < \ep< 1$, then the inequality
\begin{equation}\label{intro18}
\sum_{\substack{\bxi \in X\\\bxi \not= \bzero}} F(A\bxi) \le 8^M \ep^{-1} |X| \sum_{m=0}^M \frac{\bigl(\log \ep^{-1}|X|\bigr)^m}{m!}
\end{equation}
holds at all point $A$ in $G_1$ outside a set of $\mu_1$-measure at most $\ep$.
\end{theorem}

The inequality (\ref{intro18}) can be used to prove metric theorems of various sorts.   Here is an example. 

\begin{corollary}\label{cormet4}  Let $0 < \eta$, and let $1 < C_1$ and $1 < C_2$ be constants.  Let 
$Y_1, Y_2, \dots , Y_{\ell}, \dots $
be a sequence of finite subsets of $\Z^N$ such that
\begin{equation}\label{met1}
C_1^{\ell} \le |Y_{\ell}| \quad\text{for each $\ell = 1, 2, \dots ,$}
\end{equation}
and write
\begin{equation*}\label{met2}
\X = \bigcup_{\ell = 1}^{\infty} \big\{X \subseteq Y_{\ell}: |Y_{\ell}| \le C_2 |X|\bigr\}.
\end{equation*}
Then for $\mu_1$-almost all points $A$ in $G_1$, the inequality
\begin{equation}\label{met3}
\sum_{\substack{\bxi \in X\\\bxi \not= \bzero}} F(A\bxi) \ll_{A, \eta, M} |X| \bigl(\log 3|X|\bigr)^{M+1} \bigl(\log\log 27|X|\bigr)^{1 + \eta}
\end{equation}
holds for all subsets $X$ in $\X$.
\end{corollary}

It is instructive to examine two special cases in more detail.  If $M = 2$, $N =1$, $L_1 = L_2 = K$, and
$Y = \{0, 1, 2, \dots , K\}$, then (\ref{intro14}) asserts that
\begin{align}\label{intro19}
\begin{split}
(K + 1)&\bigl(\hh \log (K + 1)\bigr)^2 - K\\
	&\le \sum_{\substack{k=-K\\k \not= 0}}^K \min\big\{K, (2\|k\alpha_1\|)^{-1}\big\} \min\big\{K, (2\|k\alpha_2\|)^{-1}\big\}\\
	&\le \sum_{\substack{k=-K\\k \not= 0}}^K \bigl(4 \|k\alpha_1\| \|k\alpha_2\|\bigr)^{-1},
\end{split}
\end{align}
for all $\alpha_1$ and $\alpha_2$ in $\R/\Z$.  And in the special case $M = 1$, $N = 2$, $L_1 = K$ and 
\begin{equation*}\label{intro20}
Y = \bigg\{\begin{pmatrix} k_1\\
                                             k_2\end{pmatrix} : 0 \le k_1 \le K,\ 0 \le k_2 \le K\bigg\},
\end{equation*}
the lower bound (\ref{intro14}) takes the form
\begin{align}\label{intro21}
\begin{split}
(K+1)^2& \log (K+1) - K^2 \\
	&\le \underset{\bk \not= \bzero}{\sum_{k_1=-K}^K\sum_{k_2=-K}^K} \min\big\{K, \bigl(2 \|k_1\alpha_1 + k_2\alpha_2\|\bigr)^{-1}\big\}\\
	&\le \underset{\bk \not= \bzero}{\sum_{k_1=-K}^K\sum_{k_2=-K}^K} \bigl(2 \|k_1\alpha_1 + k_2\alpha_2\|\bigr)^{-1}
\end{split}
\end{align}
for all $\alpha_1$ and $\alpha_2$ in $\R/\Z$.  Theorem \ref{thmintro3} shows that there always exist points $\alpha_1$ and $\alpha_2$ in
$\R/\Z$ for which the lower bounds on the left of (\ref{intro19}) and (\ref{intro21}) cannot be significantly improved.  In both cases the
points $\alpha_1$ and $\alpha_2$ evidently depend on $K$.  This raises the question: do there exist points
$\alpha_1$ and $\alpha_2$ in $\R/\Z$ such that 
\begin{equation}\label{intro24}
\sum_{\substack{k=-K\\k \not= 0}}^K \bigl(\|k\alpha_1\| \|k\alpha_2\|\bigr)^{-1} \ll_{\alpha_1, \alpha_2} K (\log K)^2
\end{equation} 
as $K \rightarrow \infty$?  And similarly,  do there exist points $\alpha_1$ and $\alpha_2$ in $\R/\Z$ such that 
\begin{equation}\label{intro25}
\underset{\bk \not= \bzero}{\sum_{k_1=-K}^K\sum_{k_2=-K}^K} \|k_1\alpha_1 + k_2\alpha_2\|^{-1} \ll_{\alpha_1, \alpha_2} K^2 \log K
\end{equation}
as $K \rightarrow \infty$?  We will show that the inequality (\ref{intro24}) holds whenever 
the pair $\{\alpha_1, \alpha_2\}$ is a counterexample to the Littlewood conjecture, (see \cite{Venk2008} for a survey of recent work
on this conjecture.)  That is, if $\alpha_1$ and $\alpha_2$ are points in $\R/\Z$ such that
\begin{equation}\label{intro26}
0 < \liminf_{k \rightarrow \infty} k \|k \alpha_1\|\|k \alpha_2\|,
\end{equation}
then we will show that (\ref{intro24}) holds.  We will also prove a basic transference theorem which shows that
(\ref{intro26}) implies the estimate (\ref{intro25}).  In fact, we will prove generalizations of these
results to a special class of $M \times N$ matrices $A$ in the group $G_1$.

\section{Strongly badly approximable matrices}

Let $A = (\alpha_{mn})$ be an $M\times N$ matrix in the compact group $G_1 = (\R/\Z)^{MN}$, 
let $0 < \ep_m \le \h$ be real numbers for $m = 1, 2, \dots , M$, and let $K_1, K_2, \dots , K_N$, be nonnegative integers,
not all of which are zero.  A general form of Dirichlet's theorem on Diophantine approximation states that if
\begin{equation}\label{sba1}
1 \le \ep_1 \ep_2 \cdots \ep_M (K_1 + 1)(K_2 + 1) \cdots (K_N + 1),
\end{equation}
then there exists a vector $\bxi \not= \bzero$ in $\Z^N$ such that
\begin{equation}\label{sba2}
\|\alpha_{m1} \xi_1 + \alpha_{m2} \xi_2 + \cdots + \alpha_{mN} \xi_N \| \le \ep_m\quad\text{for}\ m = 1, 2, \dots , M,
\end{equation}
and
\begin{equation}\label{sba3}
|\xi_n| \le K_n\quad \text{for}\ n = 1, 2, \dots , N.
\end{equation}
Dirichlet's theorem is usually stated for matrices $A = (\alpha_{mn})$ having real entries.  However, it is obvious that (\ref{sba2})
depends only on the image of each matrix entry $\alpha_{mn}$ in $\R/\Z$.  For our purposes it is important
that the matrix $A = (\alpha_{mn})$ belongs to a compact group, and so we work with matrices $A$ in $G_1$.

We recall (see Perron \cite{P1921}, or Schmidt \cite{WMS2}) that the matrix $A = (\alpha_{mn})$ in
$G_1$ is {\it badly approximable} if there exists a positive constant $\beta = \beta(A)$ such that the inequality
\begin{equation}\label{sba7}
0 < \beta(A) \le \max_{1 \le m \le M} \bigg\{\Big\|\sum_{n=1}^N \alpha_{mn} \xi_n \Big\|\bigg\}^M
	\max_{1 \le n \le N} \big\{|\xi_n|\big\}^N
\end{equation}
holds for all points $\bxi \not= \bzero$ in $\Z^N$.   We now introduce a still more restrictive condition on $A$. 
We say that $A = (\alpha_{mn})$ in $G_1$ is {\it strongly badly approximable} if 
there exists a positive constant $\gamma = \gamma(A)$ such that the inequality
\begin{equation}\label{sba9}
0 < \gamma(A) \le \biggl(\prod_{m=1}^M \Big\|\sum_{n=1}^N \alpha_{mn} \xi_n\Big\|\biggr)
	\biggl(\prod_{n=1}^N \bigl(|\xi_n| + 1\bigr)\biggr)
\end{equation}
holds for all points $\bxi \not= \bzero$ in $\Z^N$.  It follows from the general form of Dirichlet's theorem that if $A$ is
strongly badly approximable then $0 < \gamma(A) \le 1$.  

If $M = N = 1$, then it is clear that $A = (\alpha_{11})$ is badly approximable if and only if $A$ is strongly badly approximable.  
And it is well known (see \cite[Chapter I, Corollary to Theorem IV]{Cassels1965}) that $A = (\alpha_{11})$ is badly 
approximable if and only if the partial quotients in the continued fraction expansion of $\alpha_{11}$ are bounded.  
In the remaining cases, that is, in case $M + N \ge 3$, it was shown
by O.~Perron \cite{P1921} (see also \cite[Theorem 4B]{WMS2}) that badly approximable matrices $A = (\alpha_{mn})$ exist.
However, in case $M + N \ge 3$, it is a difficult open problem to establish the existence of a strongly badly approximable 
matrix in $G_1$. 

Let $I \subseteq \{1, 2, \dots , M\}$ and $J \subseteq \{1, 2, \dots , N\}$ be nonempty subsets, and let $A(I, J)$ be the
$|I|\times|J|$ submatrix of $A$ with rows indexed by the elements of $I$ and with columns indexed by the
elements of $J$.  If $A$ is strongly badly approximable, then it follows from (\ref{sba9}) that 
\begin{equation}\label{sba10}
0 < \gamma(A) \le \biggl(\prod_{m \in I} \Big\|\sum_{n \in J} \alpha_{mn} \xi_n\Big\|\biggr)
	\biggl(\prod_{n \in J} \bigl(|\xi_n| + 1\bigr)\biggr)
\end{equation}
holds for all point $\bxi \not= \bzero$ in $\Z^N$ such that $n \mapsto \xi_n$ has support contained in $J$.  Thus each 
submatrix $A(I, J)$ is also strongly badly approximable.  In particular, each matrix entry $\alpha_{mn}$ is a badly
approximable point in $\R/\Z$.

The following result shows that if $A$ is strongly badly approximable, then an upper bound of the 
form (\ref{intro18}) holds for all subsets $X \subseteq \Z^N$ that can be written as a product of intervals, and with an implied constant
that depends only on $A$.

\begin{theorem}\label{thmsba1}  Let $A$ be a strongly badly approximable $M\times N$ matrix in the group $G_1$.
Let $K_1, K_2, \dots , K_N$, be nonnegative integers, not all of which are zero, and
\begin{equation}\label{sba15}
\K = \{\bxi \in \Z^N: 0 \le \xi_n \le K_n\ \text{for}\ n = 1, 2, \dots , N\}.
\end{equation}
Then the inequality
\begin{equation}\label{sba16}
\sum_{\substack{\bxi \in \K\\\bxi \not= \bzero}} F(A\bxi) \ll_A |\K| \bigl(\log 2|\K|\bigr)^M
\end{equation}
holds for all positive integer values of $|\K|$.
\end{theorem}

An important transference principle in Diophantine approximation (see \cite[Chapter V, Corollary to Theorem II]{Cassels1965}) 
asserts that an $M\times N$ matrix $A$ in $G_1$ is badly approximable if and only if the $N\times M$ transposed 
matrix $A^T$ is badly approximable.  Here we establish this principle for strongly badly approximable matrices.

\begin{theorem}\label{thmsba2}  Let $A = (\alpha_{mn})$ be an $M \times N$ matrix in $G_1$.  Then $A$ is strongly
badly approximable if and only if the $N\times M$ transposed matrix $A^T$ is strongly badly approximable.
\end{theorem}

In the special case $M = 1$ and $N = 2$, the statement of Theorem \ref{sba2} was proved, in a different but
equivalent form, by Cassels and Swinnerton-Dyer \cite[section 7, Lemma 5]{casselsHPFSD}.  The more general
case in which $A$ is a $1 \times N$ matrix is stated explicitly by de Mathan \cite[Theorem 1.1]{mathan2012}.  If
\begin{equation*}\label{sba17}
 A = \begin{pmatrix} \alpha_1 & \alpha_2 & \cdots & \alpha_N\end{pmatrix},
\end{equation*}
where $2 \le N$ and $1, \alpha_1, \alpha_2, \cdots , \alpha_N$, is a $\Q$-basis for a real algebraic number field of degree $N+1$,
then it follows from work of Peck \cite{peck1961} that the image 
of $A$ in $G_1$ is not strongly badly approximable.  A further refinement of this result was obtained 
by de Mathan \cite[Theorem 1.4]{mathan2012}.

 If $\alpha_1$ and $\alpha_2$ are points in $\R/\Z$ that satisfy (\ref{intro26}), then it follows immediately that the
 $2 \times 1$ matrix 
 \begin{equation*}\label{sba18}
 A = \begin{pmatrix}  \alpha_1\\    
                                      \alpha_2\end{pmatrix}
 \end{equation*}
is strongly badly approximable.  Hence the $1\times 2$ transposed matrix $A^T$ is strongly badly approximable.
Then an application of Theorem \ref{sba1} to the matrix $A^T$ establishes the bound (\ref{intro25}).  This shows
that both (\ref{intro24}) and (\ref{intro25}) hold if the pair $\{\alpha_1, \alpha_2\}$ is a counterexample to the
Littlewood conjecture.

\section{Majorizing and minorizing functions}

In this section we collect together variants of inequalities proved in \cite{V}.  And we establish new inequalities for
a special collection of trigonometric polynomials $\tau_{L-1}(x)$ defined in (\ref{tp9}).

We define three entire functions $H(z)$, $J(z)$, and $K(z)$, by 
\begin{equation*}\label{ext1}
H(z) = \left( \dfrac{\sin \pi z}{\pi}\right)^2 \biggl\{ \sum_{m=-\infty}^\infty \sgn (m) (z-m)^{-2} + 2z^{-1}\biggr\},
\end{equation*}
\begin{equation}\label{ext2}
J(z) = \hh H' (z)\ ,\ \text{ and }\ K(z) = \left( \dfrac{\sin \pi z}{\pi z}\right)^2\ .
\end{equation}
Each of these functions is real valued on the real axis and has exponential type $2\pi$. 
The functions $J$ and $K$ are integrable on $\R$ and their Fourier transforms 
\begin{equation*}\label{ext3}
\widehat{J}(t) = \int_{-\infty}^\infty J(x) e(-tx) \dx, \ \text{ and }\ \widehat{K}(t) = \int_{-\infty}^\infty K(x) e(-tx)\dx,
\end{equation*}
are continuous functions supported on $[-1,1]$.  These Fourier transforms are given explicitly (see \cite[Theorem 6]{V}) by 
\begin{align}\label{ext4}
\begin{split}
\widehat{J}(t) &= \pi t(1-|t|) \cot \pi t + |t|\ \text{ if }\ 0<|t|<1,\\
\widehat{K}(t) &= (1-|t|)\ \text{ if }\ 0\le |t|\le 1,\\
\widehat{J}(0) &= 1,\ \text{and}\ \widehat{J}(t) = \widehat K (t) =0\ \text{if}\ 1\le |t|.
\end{split}
\end{align}
It was shown in \cite[Lemma 5]{V} that the functions $H$ and $K$ satisfy the basic inequality
\begin{equation}\label{ext5}
|\sgn(x) - H(x)| \le K(x)
\end{equation}
for all real $x$.  Let $\alpha < \beta$ be real numbers and define
\begin{align}\label{ext6}
\begin{split}
\chi_{\alpha, \beta}(x) &= \hh \sgn(x - \alpha) + \hh \sgn(\beta - x)\\
	&= \begin{cases}  1& \text{if $\alpha < x < \beta$,}\\
	                              \hh& \text{if $x = \alpha$ or $x = \beta$,}\\
	                                 0& \text{if $x < \alpha$ or $\beta < x$.}\end{cases} 
\end{split}
\end{align} 
The function $\chi_{\alpha, \beta}(x)$ is the normalized characteristic function of the real interval with 
endpoints $\alpha$ and $\beta$.  That is, it satisfies
\begin{equation*}\label{ext6.1}
\chi_{\alpha, \beta}(x) = \hh \chi_{\alpha, \beta}(x-)  + \hh \chi_{\alpha, \beta}(x+)
\end{equation*}
at each real number $x$.

\begin{lemma}\label{lemext1}  Let $\alpha < \beta$ be real numbers and let $0 < \delta$.  Then there
exist real entire functions $S(z)$ and $T(z)$ of exponential type at most $2\pi\delta$, such that
\begin{equation}\label{ext7}
S(x) \le \chi_{\alpha, \beta}(x) \le T(x)
\end{equation}
for all real $x$, both $S$ and $T$ are integrable on $\R$, and their Fourier transforms
\begin{equation}\label{ext8}
\tS(y) = \int_{\R} S(x) e(-xy)\ \dx,\quad\text{and}\quad \tT(y) = \int_{\R} T(x) e(-xy)\ \dx,
\end{equation}
are both supported on the interval $[-\delta, \delta]$.  Moreover, these functions satisfy
\begin{equation}\label{ext9}
\tS(0) = \beta - \alpha - \delta^{-1},\quad\text{and}\quad \tT(0) = \beta - \alpha + \delta^{-1}.
\end{equation}
\end{lemma}

\begin{proof}  We define entire functions $S(z)$ and $T(z)$ by
\begin{equation*}\label{ext10}
S(z) = \hh H\bigl(\delta(z - \alpha)\bigr) - \hh H\bigl(\delta(z - \beta)\bigr) - \hh K\bigl(\delta(z - \alpha)\bigr) - \hh K\bigl(\delta(z - \beta)\bigr),
\end{equation*}
and
\begin{equation*}\label{ext11}
T(z) = \hh H\bigl(\delta(z - \alpha)\bigr) - \hh H\bigl(\delta(z - \beta)\bigr) + \hh K\bigl(\delta(z - \alpha)\bigr) + \hh K\bigl(\delta(z - \beta)\bigr).
\end{equation*}
Then the inequality (\ref{ext7}) follows for all real $x$ from (\ref{ext5}) and (\ref{ext6}).  Because both $H(z)$ and $K(z)$ have
exponential type $2\pi$, it is clear that $S(z)$ and $T(z)$ have exponential type at most $2\pi\delta$.

The identity
\begin{equation}\label{ext12}
\hh H\bigl(\delta(z - \alpha)\bigr) - \hh H\bigl(\delta(z - \beta)\bigr) = \delta \int_{\alpha}^{\beta} J\bigl(\delta(x - y)\bigr)\ \dy,
\end{equation}
follows from (\ref{ext2}).  It shows that the left hand side of (\ref{ext12}) is the convolution of the two integrable functions 
$x \mapsto \chi_{\alpha, \beta}(x)$ and $x \mapsto \delta J\bigl(\delta x\bigr)$,  Hence
(\ref{ext12}) is integrable on $\R$.  As $x \mapsto K\bigl(\delta x\bigr)$ is obviously integrable, we find that both
$x \mapsto S(x)$ and $x \mapsto T(x)$ are integrable on $\R$.  Then the identities in (\ref{ext9}) follow from (\ref{ext4}).
\end{proof}

\begin{corollary}\label{corext2}  Let $\alpha < \beta$ be real numbers and let $0 < \delta$.  Then there
exists an entire functions $U(z)$ of exponential type at most $\pi\delta$, such that
\begin{equation}\label{ext13}
\chi_{\alpha, \beta}(x) \le |U(x)|^2
\end{equation}
for all real $x$, and 
\begin{equation}\label{ext14}
\int_{\R} |U(x)|^2\ \dx = \beta - \alpha + \delta^{-1}.
\end{equation}
\end{corollary}

\begin{proof}  Let $T(z)$ be the real entire function in the statement of Lemma \ref{lemext1}.  As $x \mapsto T(x)$ takes
nonnegative values on $\R$, it follows from a theorem of Fej\'er (see \cite[pp. 124--126]{Boas}) that there exists an
entire function $U(z)$ of exponential type at most $\pi\delta$ such that
\begin{equation*}\label{15}
T(z) = U(z) U^*(z),
\end{equation*}
where $U^*(z) = \overline{U(\overline{z})}$.  In particular, we have
\begin{equation*}\label{16}
T(x) = U(x) U^*(x) = |U(x)|^2
\end{equation*}
for all real $x$.  Now (\ref{ext13}) and (\ref{ext14}) follow from (\ref{ext7}) and the identity on the right hand side of (\ref{ext9}),
respectively.
\end{proof}

For each positive integer $L$ we write $J_{L+1}(z) = (L+1)J((L+1)z)$, so that $J_{L+1}(z)$ has exponential type $2\pi (L+1)$. 
Then for all real $t$ the Fourier transforms $\widehat J$ and $\widehat J_{L+1}$ are related by the identity 
\begin{equation*}\label{ext25}
\widehat J((L+1)^{-1}t) = \widehat J_{L+1} (t).
\end{equation*}
Similar remarks apply to $K$ and $K_{L+1}$.  Using this notation we define trigonometric polynomials 
$j_L(x)$ and $k_L(x)$ by 
\begin{equation}\label{ext28}
j_L (x) = \sum_{m=-\infty}^\infty J_{L+1}(x+m) = \sum_{\ell=- L}^L \widehat J_{L+1}(\ell) e(\ell x),
\end{equation}
and 
\begin{equation}\label{ext29}
k_L (x) = \sum_{m=-\infty}^\infty K_{L+1} (x+m) = \sum_{\ell=-L}^L \widehat K_{L+1}(\ell) e(\ell x)\ .
\end{equation}
The identities on the right of (\ref{ext28}) and (\ref{ext29}) follow from the Poisson summation formula. 
We also define the periodic function $x \mapsto\psi(x)$ by 
\begin{equation*}\label{ext30}
\psi (x) = \begin{cases} x- [x] -\hh &\text{if $x \notin \Z$,}\\ 
			              0 &\text{if $x \in \Z$.}\end{cases}
\end{equation*}
The trigonometric polynomials 
\begin{equation}\label{ext31}
\psi * j_L(x) = \int_{-1/2}^{1/2} \psi (x-y) j_L (y)\,dy 
	= \sum_{\substack{\ell = -L\\ \ell \not= 0}}^L (-2\pi i\ell)^{-1} \widehat{J}_{L+1}(\ell) e(\ell x), 
\end{equation}
and $k_L(x)$, satisfy
\begin{equation}\label{ext32}
\sgn\bigl(\psi*j_L(x)\bigr) = \sgn\bigl(\psi(x)\bigr),
\end{equation}
{ }
\begin{align}\label{ext33}
\begin{split}
2\bigl|\psi (x) - \psi *j_L(x)\bigr| &\le (L+1)^{-1} k_L(x)\\
						&= (L+1)^{-2} \biggl(\frac{\sin \pi(L+1) x}{\sin \pi x}\biggr)^2,
\end{split}
\end{align}
and
\begin{equation}\label{ext34}
\bigl|\psi*j_L(x)\bigr| \le \bigl|\psi(x)\bigr|,
\end{equation}
for all $x$ in $\R/\Z$.  A proof of (\ref{ext32}), (\ref{ext33}), and (\ref{ext34}), is given in \cite[Theorem~18]{V}. 

If $\alpha < \beta < \alpha +1$ we define the normalized characteristic function of an interval in $\R/\Z$ with endpoints $\alpha$ and $\beta$ by
\begin{equation*}\label{ext40}
\varphi_{\alpha, \beta}(x) = \begin{cases}  1& \text{if $\alpha < x - n < \beta$ for some $n \in \Z$,}\\
                                                         \h & \text{if $\alpha - x \in \Z$ or if $\beta - x \in \Z$,}\\
                                                         0 & \text{otherwise.} \end{cases}
\end{equation*}
The periodic functions $\varphi_{\alpha, \beta}(x)$ and $\psi (x)$ are related by the elementary identity 
\begin{equation}\label{ext43}
\varphi_{\alpha, \beta}(x) = (\beta - \alpha) + \psi (\alpha - x) + \psi (x - \beta),
\end{equation}
which is a periodic analogue of (\ref{ext6}).  By combining (\ref{ext33}) and (\ref{ext43}) we obtain the inequality 
\begin{align}\label{ext45}
\begin{split} 
|\varphi_{\alpha, \beta}&(x) - \varphi_{\alpha, \beta} * j_L (x)|\\
&\qquad \le |\psi (\alpha - x) - \psi *j_L(\alpha - x)| + |\psi (x - \beta) - \psi *j_L(x - \beta)|\\ 
&\qquad \le (2L+2)^{-1} \{ k_L (\alpha - x) + k_L(x - \beta)\}
\end{split}
\end{align}
for all $x$ in $\R/\Z$.  Alternatively, (\ref{ext45}) follows directly from \cite[Theorem~19]{V}. 

\begin{lemma}\label{lemext3}  Let $\alpha < \beta < \alpha + 1$ and let $1 \le L$ be an integer.  Then there exist real trigonometric 
polynomials  $s(x)$ and $t(x)$ of degree at most $L$, such that
\begin{equation}\label{ext48}
s(x) \le \p_{\alpha, \beta}(x) \le t(x)
\end{equation}
at each point $x$ in $\R/\Z$.  Moreover, the Fourier coefficients $\ts(0)$ and $\ttt(0)$ are given by
\begin{equation}\label{ext49}
\ts(0) = \beta - \alpha - (L + 1)^{-1}\quad\text{and}\quad \ttt(0) = \beta - \alpha + (L+1)^{-1}.
\end{equation}
\end{lemma}

\begin{proof}  We define trigonometric polynomials of degree $L$ by
\begin{equation*}\label{ext50}
s(x) = \varphi_{\alpha, \beta}*j_L(x) - (2L+2)^{-1}\{k_L(\alpha - x) + k_L(x - \beta)\},
\end{equation*}
and
\begin{equation*}\label{ext51}
t(x) = \varphi_{\alpha, \beta}*j_L(x) + (2L+2)^{-1}\{k_L(\alpha - x) + k_L(x - \beta)\}.
\end{equation*}
The inequality (\ref{ext48}) follows immediately from (\ref{ext45}).  The Fourier coefficient $\ts(0)$ is then
\begin{align*}\label{ext52}
\begin{split}
\ts(0) &= \widehat{\varphi}_{\alpha, \beta}(0)\widehat{j}_L(0) - (2L+2)^{-1}\{\tk(0) + \tk(0)\}\\
	 &= \beta - \alpha - (L+1)^{-1},
\end{split}
\end{align*}
and similarly for $\ttt(0)$.  This verifies the identity (\ref{ext49}).
\end{proof}

A result closely related to Lemma \ref{lemext3} is obtained in \cite[Lemma~5]{BMV}.

\begin{corollary}\label{corext4}  Let $\alpha < \beta < \alpha + 1$ and let $1 \le L$ be an integer.  Then there exists a trigonometric 
polynomials
\begin{equation}\label{ext60}
u(x) = \sum_{\ell = 0}^L \tu(\ell) e(\ell x),
\end{equation}
such that
\begin{equation}\label{ext61}
\p_{\alpha, \beta}(x) \le |u(x)|^2
\end{equation}
at each point $x$ in $\R/\Z$, and
\begin{equation}\label{ext62}
\int_{R/\Z} |u(x)|^2\ \dx = \beta - \alpha + (L + 1)^{-1}.
\end{equation} 
\end{corollary}

\begin{proof}  Let $t(x)$ be the trigonometric polynomial of degree at most $L$ that occurs in the statement of Lemma \ref{lemext3}.
Because $x \mapsto t(x)$ takes nonnegative values, Fejer's theorem for trigonometric polynomials (see \cite{Fejer1916})
establishes the existence of a trigonometric polynomial $u(x)$ of the form (\ref{ext60}), such that
\begin{equation*}\label{ext65}
t(x) = |u(x)|^2
\end{equation*}
at each point $x$ in $\R/\Z$.  The inequality (\ref{ext61}) and the identity (\ref{ext62}) follow immediately from the lemma.
\end{proof}

Because $x \mapsto \psi*j_L(x)$ is an odd, trigonometric polynomial, we have $\psi*j_L(0) = 0$ and
$\psi*j_L\bigl(\h\bigr) = 0$.  Therefore the function $x \mapsto \tau_{L-1}(x)$ defined by
\begin{equation}\label{tp9}
\tau_{L-1}(x) = \biggl(\frac{-2\pi}{\sin 2\pi x}\biggr) \psi*j_L(x) = \frac{\psi*j_L(x)}{\psi*j_1(x)}
\end{equation}
is a trigonometric polynomial of degree $L-1$.  Then (\ref{ext32}) implies that
\begin{equation*}\label{tp10}
\sgn\bigl(\psi*j_L(x)\bigr) = \sgn\bigl(\psi(x)\bigr) = \sgn\bigl(\psi*j_1(x)\bigr).
\end{equation*}
It follows that 
\begin{equation}\label{tp11}
0 < \tau_{L-1}(x)
\end{equation}
for all $x$ in $\R/\Z$ such that $x \not= 0$ and $x\not= \h$.  We also have
\begin{equation*}\label{tp12}
-\Bigl(\frac{{\rm d}}{\dx}\Bigr) \psi*j_L(x) \Bigl|_{x = 0} = L,
\end{equation*}
which easily leads to the identity
\begin{equation}\label{tp13}
\tau_{L-1}(0) = L.
\end{equation}

The Fourier coefficients $\ttau_{L-1}(\ell)$ can be determined explicitly by writing
\begin{equation*}\label{tp14}
\psi*j_L(x) = - \sum_{\ell=1}^L \biggl(\frac{1}{\pi \ell}\biggr) \tJ\biggl(\frac{\ell}{L+1}\biggr) \sin 2\pi \ell x.
\end{equation*}
Then we have
\begin{align}\label{tp15}
\begin{split}
\tau_{L-1}(x) &= \sum_{\ell=1}^L \biggl(\frac{2}{\ell}\biggr) \tJ\biggl(\frac{\ell}{L+1}\biggr)\bigg\{\frac{\sin 2\pi \ell x}{\sin 2\pi x}\bigg\}\\
   &= \sum_{\ell=0}^{\infty} \biggl(\frac{2}{2\ell+1}\biggr)\tJ\biggl(\frac{2\ell+1}{L+1}\biggr)\\
   &\qquad + \sum_{\ell=1}^{\infty} \biggl(\frac{4}{2\ell+1}\biggr)\tJ\biggl(\frac{2\ell+1}{L+1}\biggr) \sum_{m=1}^\ell \cos 2\pi(2m)x \\
   &\qquad\qquad + \sum_{\ell=1}^{\infty} \biggl(\frac{2}{\ell}\biggr)\tJ\biggl(\frac{2\ell}{L+1}\biggr) \sum_{m=1}^\ell \cos 2\pi(2m-1)x,
\end{split}
\end{align}
where the sums on the index $\ell$ contain only finitely many nonzero terms (because $t\mapsto \tJ(t)$ is supported on 
$[-1, 1]$.)  After further reorganization we arrive at the identity
\begin{equation}\label{tp16}
\begin{split}
\tau_{L-1}(x) &= \sum_{\ell=0}^{\infty} \biggl(\frac{2}{2\ell+1}\biggr)\tJ\biggl(\frac{2\ell+1}{L+1}\biggr)\\
   &\qquad + \sum_{m=1}^{L-1} 
	\Bigg\{\sum_{\ell=0}^{\infty} \biggl(\frac{4}{2\ell+m+1}\biggr)\tJ\biggl(\frac{2\ell+m+1}{L+1}\biggr)\Bigg\} \cos 2\pi m x.
\end{split}
\end{equation}
As
\begin{equation*}\label{tp16.5}
\tau_{L-1}(x) = \ttau_{L-1}(0) + 2 \sum_{m=1}^{L-1} \ttau_{L-1}(m) \cos 2\pi m x,
\end{equation*}
it is clear from (\ref{tp16}) that 
\begin{equation}\label{tp17}
0 < \ttau_{L-1}(m)\quad\text{for}\quad |m| \le L-1.
\end{equation}
Therefore we conclude from the definition (\ref{tp9}) that
\begin{equation}\label{tp18}
\sup\big\{\tau_{L-1}(x): x\in\R/\Z\big\} = \tau_{L-1}(0) = L.
\end{equation}

We collect together useful properties of $\tau_{L-1}(x)$ in the following lemma.

\begin{lemma}\label{lemtp5}  For each positive integer $L$ the trigonometric polynomial
\begin{equation*}\label{tp19}
\tau_{L-1}(x) = \biggl(\frac{-2\pi}{\sin 2\pi x}\biggr) \psi*j_L(x) = \frac{\psi*j_L(x)}{\psi*j_1(x)}
\end{equation*}
has degree $L - 1$, and takes nonnegative values on $\R/\Z$.  The inequality
\begin{equation}\label{tp20}
\tau_{L-1}(x) \le \min\big\{L, \bigl(2\|x\|\bigr)^{-1}\big\}, 
\end{equation} 
holds for all $x$ in $\R/\Z$, and there is equality in the inequality {\rm (\ref{tp20})} at $x = 0$.  The Fourier coefficients 
$\ttau_{L-1}(\ell)$ are positive for $|\ell| \le L-1$, and
\begin{equation}\label{tp21}
\log (L+1) \le \ttau_{L-1}(0) \le 1 + \log L.
\end{equation}
\end{lemma}

\begin{proof}  The inequality (\ref{tp11}) implies that $\tau_{L-1}(x)$ takes nonnegative values  on $\R/\Z$, and (\ref{tp16})
verifies that its Fourier coefficients are positive for $|\ell| \le L-1$.  From the inequality (\ref{ext34}) and the definition (\ref{tp9}),
we conclude that
\begin{equation}\label{tp23}
2 \|x\| \tau_{L-1}(x) \le \sup\Big\{\frac{4\pi y(\h - y)}{\sin 2\pi y}: 0 < y < \hh\Big\} = 1.
\end{equation}
By combining (\ref{tp18}) and (\ref{tp23}) we obtain the basic inequality (\ref{tp20}), and we also get the case of equality in
(\ref{tp20}) at $x = 0$.  

To establish the inequality on the right of (\ref{tp21}), we integrate both sides of (\ref{tp20}) over $\R/\Z$.  

The inequality on the left of (\ref{tp21}) is trivial to check in case $L = 1$ and $L=2$.  Therefore in what follows we will
assume that $3 \le L$.  It will be useful to define the function $w: \R/\Z \rightarrow [0, \infty)$ by
$w(0) = 0$, $w\bigl(\h\bigr) = 1$, and
\begin{equation}\label{tp24}
w(x) = \frac{-2\pi \psi(x)}{\sin 2\pi x},\quad\text{if $x \not= 0$ and $x \not= \hh$.}
\end{equation}
It follows easily that $x \mapsto w(x)$ is continuous and positive on $\R/\Z \setminus \{0\}$.  The inequality
(\ref{ext33}) leads to the identity
\begin{equation}\label{tp25}
\psi\bigl(\tfrac{\ell}{L+1}\bigr) = \psi*j_L\bigl(\tfrac{\ell}{L+1}\bigr)
\end{equation}
for $\ell = 1, 2, \dots , L$.  Then (\ref{tp24}) and (\ref{tp25}) imply that
\begin{equation}\label{tp26}
w\bigl(\tfrac{\ell}{L+1}\bigr) = \tau_{L-1}\bigl(\tfrac{\ell}{L+1}\bigr)
\end{equation}
for $\ell = 1, 2, \dots , L$.  Because $\tau_{L-1}(x)$ is a trigonometric polynomial of degree $L-1$, we have
\begin{align}\label{tp28}
\begin{split}
(L + 1) \ttau_{L-1}(0) &= \sum_{m = 1-L}^{L-1} \ttau_{L-1}(m) \sum_{\ell = 0}^L e\bigl(\tfrac{\ell m}{L+1}\bigr)\\
	                            &= \sum_{\ell = 0}^L \tau_{L-1}\bigl(\tfrac{\ell}{L+1}\bigr)\\
	                            &= \tau_{L-1}(0) + \sum_{\ell = 1}^L \tau_{L-1}\bigl(\tfrac{\ell}{L+1}\bigr).
\end{split}
\end{align}
We now combine (\ref{tp13}), (\ref{tp26}), and (\ref{tp28}), to obtain the identity
\begin{equation}\label{tp30}
(L + 1) \ttau_{L-1}(0) = L + \sum_{\ell = 1}^L w\bigl(\tfrac{\ell}{L+1}\bigr).
\end{equation}

On the open interval $(0, 1)$ the continuous function $x \mapsto w(x)$ can be expressed as a power series in $(x - \h)^2$
with positive coefficients.  Hence the function is convex and satisfies $w(x) = w(1 - x)$.  Then it follows from Simpson's rule that
\begin{align}\label{tp31}
\begin{split}
\sum_{\ell = 1}^L w\bigl(\tfrac{\ell}{L+1}\bigr)
	&\ge w\bigl(\tfrac{1}{L+1}\bigr) + (L+1) \int_{\tfrac{1}{L+1}}^{\tfrac{L}{L+1}} w(x)\ \dx\\
	&= w\bigl(\tfrac{1}{L+1}\bigr) + (L+1) \int_{\hh - \tfrac{1}{L+1}}^{\hh + \tfrac{1}{L+1}} w(x)\ \dx\\
	&\qquad\qquad + (L+1) \int_{\tfrac{1}{L+1}}^{\hh - \tfrac{1}{L+1}} \bigl(w(x) + w\bigl(\hh + x\bigr)\bigr)\ \dx.
\end{split}
\end{align}
From (\ref{tp24}) we get
\begin{equation}\label{tp32}
w\bigl(\tfrac{1}{L+1}\bigr) \ge (L+1)\bigl(\hh - \tfrac{1}{L+1}\bigr) = \hh(L+1) - 1.
\end{equation}
The minimum of the even, convex function $x \mapsto w(x)$ on $(0, 1)$ occurs at $w(\h) = 1$, and it follows that
\begin{equation}\label{tp33}
 (L+1) \int_{\hh - \tfrac{1}{L+1}}^{\hh + \tfrac{1}{L+1}} w(x)\ \dx \ge 2.
\end{equation}
For the remaining integral on the right of (\ref{tp31}) we find that
\begin{align}\label{tp34}
\begin{split}
(L+1) \int_{\tfrac{1}{L+1}}^{\hh - \tfrac{1}{L+1}}&\bigl(w(x) + w\bigl(\hh + x\bigr)\bigr)\ \dx\\
	&= \pi(L+1) \int_{\tfrac{1}{L+1}}^{\hh - \tfrac{1}{L+1}} |\csc 2\pi x|\ \dx\\
	&= \pi (L+1) \int_{\tfrac{2}{L+1}}^{\hh} |\csc \pi x|\ \dx\\
	&\ge (L+1) \int_{\tfrac{2}{L+1}}^{\hh} x^{-1}\ \dx\\
	&= (L+1) \log(L+1) - (L+1)(2 \log 2).
\end{split}
\end{align}
Finally, we combine the estimates (\ref{tp30}), (\ref{tp31}), (\ref{tp32}), (\ref{tp33}), and (\ref{tp34}), to obtain the lower bound
\begin{equation}\label{tp35}
(L+1) \ttau_{L-1}(0) \ge (L+1)\bigl(\tfrac{3}{2} - 2\log 2\bigr) + (L+1) \log (L+1).
\end{equation}
Clearly (\ref{tp35}) verifies the inequality on the left of (\ref{tp21}).
\end{proof}

\section{Large sieve inequalities}

For each positive integer $M$, we define $P_M:(\R/\Z)^M \rightarrow [0, 1]$ by
\begin{equation}\label{ls1}
P_M(\bx) = \prod_{m=1}^M \|x_m\|.
\end{equation}
And for each positive integer $N$, we define $Q_N: \Z^N \rightarrow \{1, 2, \dots \}$ by
\begin{equation}\label{ls2}
Q_N(\bwy) = \prod_{n=1}^N \bigl(|y_n| + 1\bigr).
\end{equation}
Let $K_1, K_2, \dots , K_N$, and $L_1, L_2, \dots , L_M$, be two sets of nonnegative integers.  We define corresponding subsets
$\K \subseteq \Z^N$ and $\lL \subseteq \Z^M$ by
\begin{equation}\label{ls3}
\K = \{\bxi \in \Z^N: 0 \le \xi_n \le K_n\ \text{for}\ n = 1, 2, \dots , N\}
\end{equation}
and
\begin{equation}\label{ls4}
\lL = \{\ceta \in \Z^N: 0 \le \eta_m \le L_m\ \text{for}\ m = 1, 2, \dots , M\}.
\end{equation}
It follows that the difference set $\K - \K$ is 
\begin{equation*}\label{ls5}
\K - \K = \{\bxi \in \Z^N : |\xi_n| \le K_n\ \text{for}\ n = 1, 2, \dots , N\},
\end{equation*}
and similarly for $\lL - \lL$.  We write $|\K|$ and $|\lL|$ for the cardinality of $\K$ and $\lL$, respectively.

Using $A$ in $G_1$, and the subsets $\K \subseteq \Z^N$ and $\lL \subseteq \Z^M$, we define
\begin{equation}\label{ls7}
\gamma(A, \K) = \min \big\{P_M(A\bxi) Q_N(\bxi) : \bxi \in \K - \K,\ \text{and}\ \bxi \not= \bzero\big\}.
\end{equation}
It follows from the inequalities (\ref{sba1}),
(\ref{sba2}), and (\ref{sba3}), in the statement of Dirichlet's theorem that $0 \le \gamma(A, \K) \le 1$.

\begin{lemma}\label{lemls1}  If $A$ is a point in $G_1$ such that $0 < \gamma(A, \K)$, then for all complex valued functions
$\ceta \mapsto b(\ceta)$ defined on $\lL$, we have
\begin{equation}\label{ls22}
\sum_{\bxi \in \K}~\biggl|\sum_{\ceta \in \lL} b(\ceta) e\bigl(\ceta^T A \bxi\bigr)\biggr|^2 
	\le \Bigl(\gamma(A, \K)^{-\frac1M} |\K|^{\frac1M} + |\lL|^{\frac1M}\Bigr)^M \sum_{\ceta \in \lL} |b(\ceta)|^2.
\end{equation}
\end{lemma}

\begin{proof}  Let $\delta_1, \delta_2, \dots , \delta_M$, be real numbers such that
$0 < \delta_m < 1$ and
\begin{equation}\label{ls23}
\biggl(\prod_{m=1}^M \delta_m\biggr)\biggl(\prod_{n=1}^N (K_n + 1)\biggr) = \theta \gamma(A, \K),
\end{equation}
where $0 < \theta < 1$.  By Corollary \ref{corext2}, for each $m = 1, 2, \dots , M$, there exists an entire function $U_m(z)$ of 
exponential type at most $\pi\delta_m$ such that
\begin{equation}\label{ls24}
1 \le |U_m(x)|^2
\end{equation}
for all real $x$ satisfying $0 \le x \le L_m$, and
\begin{equation}\label{ls25}
\int_{\R} |U_m(x)|^2\ \dx = L_m + \delta_m^{-1}.
\end{equation}
It follows that the function
\begin{equation*}\label{ls26}
U(\bx) = \prod_{m=1}^M U_m(x_m)
\end{equation*}
belongs to $L^2\bigl(\R^M\bigr)$, and the Fourier transform $\tU(\bwy)$ is supported on the compact subset
\begin{equation*}\label{ls27}
E = \big\{\bwy \in \R^M : |y_m| \le \hh\delta_m\ \text{for each}\ m = 1, 2, \dots , M\big\}.
\end{equation*}
Moreover, (\ref{ls24}) implies that the inequality
\begin{equation*}\label{ls28}
1 \le \bigl|U(\ceta)\bigr|^2
\end{equation*}
holds at each point $\ceta$ in $\lL$.  It follows that the trigonometric polynomials
\begin{equation*}\label{ls29}
B(\bx) = \sum_{\ceta \in \lL} b(\ceta) e\bigl(\ceta^T \bx\bigr)
\end{equation*}
and
\begin{equation*}\label{ls30}
\widetilde{B}(\bx) = \sum_{\ceta \in \lL} b(\ceta) U(\ceta)^{-1} e\bigl(\ceta^T \bx\bigr),
\end{equation*}
are related by the identity
\begin{equation*}\label{ls31}
B(\bx) = \int_{\R^M} \tU(\bwy) \widetilde{B}(\bx + \bwy)\ \dbwy.
\end{equation*}
As the Fourier transform $\tU(\bwy)$ is supported on the subset $E \subseteq \R^M$, using
(\ref{ls25}) and Cauchy's inequality we find that
\begin{align}\label{ls32}
\begin{split}
\bigl|B(\bx)\bigr|^2 &\le \int_E \bigl|\tU(\bwy)\bigr|^2\ \dbwy \int_E \bigl|\widetilde{B}(\bx + \bw)\bigr|^2\ \dbw\\
	&= \prod_{m=1}^M \bigl(L_m + \delta_m^{-1}\bigr) \int_{\bx + E} \bigl|\widetilde{B}(\bw)\bigr|^2\ \dbw.
\end{split}
\end{align}
Therefore the inequality (\ref{ls32}) implies that
\begin{equation}\label{ls36}
\sum_{\bxi \in \K} \bigl|B(A\bxi)\bigr|^2 
	\le \prod_{m=1}^M \bigl(L_m + \delta_m^{-1}\bigr) \sum_{\bxi \in \K} \int_{A\bxi + E} \bigl|\widetilde{B}(\bw)\bigr|^2\ \dbw.
\end{equation}

We claim that the subsets in the collection
\begin{equation}\label{ls37}
\big\{A\bxi + \bv + E : \bxi \in \K\ \text{and}\ \bv \in \Z^M\big\}
\end{equation}
are disjoint subsets of $\R^M$.  Suppose that $\bxi_1$ and $\bxi_2$ are points in $\K$, $\bv_1$ and $\bv_2$ are 
points in $\Z^M$, $\be_1$ and $\be_2$ are points in $E$, and
\begin{equation}\label{ls38}
A\bxi_1 + \bv_1 + \be_1 = A\bxi_2 + \bv_2 + \be_2.
\end{equation}
We consider two cases.  First we suppose that $\bxi_1 \not= \bxi_2$.  Then (\ref{ls38}) implies that that for each $m = 1, 2, \dots , M$, we have
\begin{align}\label{ls39}
\begin{split}
\bigg\| \sum_{n=1}^N \alpha_{mn} (\xi_{n1} - \xi_{n2})\bigg\| 
	&\le \biggl| \sum_{n=1}^N \alpha_{mn} (\xi_{n1} - \xi_{n2}) + v_{1m} - v_{2m}\biggr|\\
	&= |e_{m1} - e_{m2}|\\
	&\le \delta_m.
\end{split}
\end{align}
As $\bxi_1 - \bxi_2$ is a nonzero point in $\K - \K$, (\ref{ls23}) and (\ref{ls39}) imply that
\begin{equation*}\label{ls40}
P_M\bigl(A(\bxi_1 - \bxi_2)\bigr) Q_N(\bxi_1 - \bxi_2) \le \theta \gamma(A, \K) < \gamma(A, \K), 
\end{equation*}
which is impossible.  Next we suppose that $\bxi_1 = \bxi_2$.  In this case we get
\begin{equation}\label{ls41}
|v_{1m} - v_{2m}| = |e_{1m} - e_{2m}| \le \delta_m < 1
\end{equation}
for each $m = 1, 2, \dots , M$.  But (\ref{ls41}) implies that $\bv_1 = \bv_2$, and therefore $\be_1 = \be_2$.  We have shown 
that (\ref{ls38}) implies that $\bxi_1 = \bxi_2$, $\be_1 = \be_2$, and $\bv_1 = \bv_2$, and this verifies our claim.

Because the subsets in the collection (\ref{ls37}) are disjoint in $\R^M$, the images of the subsets
\begin{equation*}\label{41.5}
\big\{A\bxi + E : \bxi \in \K\big\}
\end{equation*}
in the group $G_2 = (\R/\Z)^M$ are also disjoint.  In particular, the set
\begin{equation}\label{ls42}
\bigcup_{\bxi \in \K} \bigl(A\bxi + E\bigr)
\end{equation}
is contained in a fundamental domain for the quotient group $(\R/\Z)^M$.  Therefore we obtain the estimate
\begin{align}\label{ls43}
\begin{split}
\sum_{\bxi \in \K} \int_{A\bxi + E} \bigl|\widetilde{B}(\bw)\bigr|^2\ \dbw &\le \int_{(\R/\Z)^M} \bigl|\widetilde{B}(\bw)\bigr|^2\ \dbw\\
	&= \sum_{\ceta \in \lL} \bigl|b(\ceta)\bigr|^2\bigl|U(\ceta)\bigr|^{-2}\\
	&\le  \sum_{\ceta \in \lL} \bigl|b(\ceta)\bigr|^2.
\end{split}
\end{align} 
By combining (\ref{ls36}) and (\ref{ls43}), we arrive at the inequality
\begin{equation}\label{ls44}
\sum_{\bxi \in \K}~\biggl|\sum_{\ceta \in \lL} b(\ceta) e(\ceta^TA\bxi)\biggr|^2 
	\le \prod_{m=1}^M \bigl(L_m + \delta_m^{-1}\bigr) \sum_{\ceta \in \lL} \bigl|b(\ceta)\bigr|^2.
\end{equation}

For each $m = 1, 2, \dots , M$, we select $\delta_m$ so that
\begin{equation}\label{ls45}
\delta_m^{-1} = X(L_m + 1) + 1,
\end{equation}
where $X$ is a positive real parameter at our disposal.  This clearly verifies the requirement that
$0 < \delta_m < 1$.  We select $X$ to be the unique positive real number such that
\begin{equation}\label{ls46}
\prod_{m=1}^M \bigl(X + (L_m + 1)^{-1}\bigr) = \frac{|\K|}{\theta \gamma(A, \K) |\lL|}.
\end{equation}
Then it is obvious that this choice of $X$ must satisfy the inequality
\begin{equation}\label{ls47}
X^M \le \frac{|\K|}{\theta \gamma(A, \K) |\lL|}.
\end{equation}
And it follows easily that the identity (\ref{ls23}) holds.  Using (\ref{ls45}) and (\ref{ls47}) we find that
\begin{align}\label{ls48}
\begin{split}
\prod_{m=1}^M \bigl(L_m + \delta_m^{-1}\bigr) &= |\lL| (X + 1)^M\\
	&\le| \lL| \biggl(\biggl(\frac{|\K|}{\theta \gamma(A, \K) |\lL|}\biggr)^{\frac1M} + 1\biggr)^M\\
	& = \Bigl((\theta \gamma(A, \K))^{-\frac1M} |\K|^{\frac1M} + |\lL|^{\frac1M}\Bigr)^M.
\end{split}
\end{align}
As $0 < \theta < 1$ is arbitrary, the inequality (\ref{ls22}) follows from (\ref{ls44}) and (\ref{ls48}).
\end{proof}

\begin{lemma}\label{lemls3}  Assume that $A$ is a point in $G_1$ such that $0 < \gamma(A, \K)$, and $\bwy$ is a point in $(\R/\Z)^M$.
Let $0 < \delta_m \le 1$ for each $m = 1, 2, \dots , M$, and let $\bx \mapsto \Phi\bigl(\bx, \Delta(\bwy)\bigr)$ denote 
the characteristic function of the subset
\begin{equation}\label{ls100} 
\Delta(\bwy) = \{\bx \in (\R/\Z)^M : \|x_m - y_m\| \le \hh \delta_m\ \text{for each}\ m = 1, 2, \dots , M\}.
\end{equation}
Then for each subset $\K$ of the form {\rm (\ref{ls3})}, we have
\begin{equation}\label{ls101}
\sum_{\bxi \in \K} \Phi\bigl(A\bxi, \Delta(\bwy)\bigr) \le 4^M \gamma(A, \K)^{-1} \mu_2\bigl(\Delta(\bwy)\bigr) |\K| + 6^M,
\end{equation}
where $\mu_2\bigl(\Delta(\bwy)\bigr) = \delta_1 \delta_2 \cdots \delta_M$ is the normalized Haar measure
of $\Delta(\bwy) \subseteq (\R/\Z)^M$.
\end{lemma}

\begin{proof}  For each $m = 1, 2, \dots , M$, let $x \mapsto \chi_m(x)$ be the characteristic function of the subset
\begin{equation*}\label{103}
\big\{x \in \R/\Z : \|x - y_m\| \le \hh\delta_m\big\}.
\end{equation*} 
We apply Corollary \ref{corext4} to the nonnegative valued function $\chi_m(x)$ with $L = L_m$.
By that result there exists a trigonometric polynomial
\begin{equation*}\label{ls108}
u_m(x) = \sum_{\ell = 0}^{L_m} \tu_m(\ell) e(\ell x)
\end{equation*}
such that, if $\|x - y_m\| \le \h\delta_m$ then
\begin{equation*}\label{ls109}
1 \le |u_m(x)|^2,
\end{equation*}
and 
\begin{equation*}\label{ls110}
\int_{\R/\Z} |u_m(x)|^2\ \dx = \delta_m + (L_m + 1)^{-1}.
\end{equation*}
Now let $\lL \subseteq \Z^M$ be defined by (\ref{ls4}).  It follows that there exist complex numbers $b(\ceta)$, defined 
at each point $\ceta$ in $\lL$, such that
\begin{equation*}\label{ls111}
\Phi\bigl(\bx, \Delta(\bwy)\bigr) \le \prod_{m=1}^M \bigl|u_m(x_m)\bigr|^2 = \biggl|\sum_{\ceta \in \lL} b(\ceta) e(\ceta^T \bx)\biggr|^2
\end{equation*}
at each point $\bx$ in $(\R/\Z)^M$.   The inequality (\ref{ls22}) implies that
\begin{align}\label{ls112}
\begin{split}
\sum_{\bxi \in \K} \Phi\bigl(A\bxi, \Delta(\bwy)\bigr) &\le \sum_{\bxi \in \K}~\biggl|\sum_{\ceta \in \lL} b(\ceta) e\bigl(\ceta^T A \bxi\bigr)\biggr|^2\\ 
	          &\le \Bigl(\gamma(A, \K)^{-\frac1M} |\K|^{\frac1M} + |\lL|^{\frac1M}\Bigr)^M \sum_{\ceta \in \lL} |b(\ceta)|^2\\
	          &\le 2^{M-1} \bigl(\gamma(A, \K)^{-1} |\K| + |\lL|\bigr) \prod_{m=1}^M \bigl(\delta_m + (L_m + 1)^{-1}\bigr).                             
\end{split}
\end{align}
The positive integers $L_m$ are at our disposal.  We select
\begin{equation*}\label{ls113}
L_m = \bigl[\delta_m^{-1}\bigr] \le \delta_m^{-1} < L_m + 1,
\end{equation*}
and the upper bound
\begin{align}\label{ls114}
\begin{split}
2^{M-1} \bigl(\gamma(A, \K)^{-1} |\K| + |\lL|\bigr) &\prod_{m=1}^M \bigl(\delta_m + (L_m + 1)^{-1}\bigr)\\
	&\le 4^M \gamma(A, \K)^{-1} \mu_2\bigl(\Delta(\bwy)\bigr) |\K| + 6^M
\end{split}
\end{align}
follows easily.  Then (\ref{ls101}) follows from (\ref{ls112}) and (\ref{ls114}).
\end{proof}

\section{Proof of Theorem \ref{thmintro1} and Corollary \ref{corintro2}}

We define a positive integer valued function $\bxi \mapsto v(\bxi)$ on elements $\bxi$ in the difference set $X = Y - Y$ by
\begin{equation}\label{prf1}
\biggl|\sum_{\ceta \in Y} e(\bx^T \ceta)\biggr|^2 = \sum_{\bxi \in X} v(\bxi) e(\bx^T \bxi).
\end{equation}
It follows that $1 \le v(\bxi) \le |Y|$ at each point $\bxi$ in $X$, $v(\bzero) = |Y|$, and
\begin{equation*}\label{prf2}
\sum_{\bxi \in X} v(\bxi) = |Y|^2.
\end{equation*}
We also define a trigonometric polynomial $\sigma: (\R/\Z)^M \rightarrow [0, \infty)$ by
\begin{equation*}\label{prf3}
\sigma(\bx) = \prod_{m=1}^M \tau_{L_m - 1}(x_m),
\end{equation*}
where $\tau_{L_m - 1}(x)$ is the nonnegative trigonometric polynomial defined by (\ref{tp9}).  It follows from (\ref{tp17})
that $\sigma(\bx)$ has positive Fourier coefficients supported on the subset
\begin{equation*}\label{prf4}
\lL - \lL = \big\{\bell \in \Z^M : 1-L_m \le \ell_m \le L_m - 1\big\}.
\end{equation*}
From the statement of Lemma \ref{lemtp5}, we conclude that
\begin{equation}\label{prf7}
0 \le \sigma(\bx) \le \prod_{M=1}^M \min\big\{L_m, (2\|x_m\|)^{-1}\big\} = F_{\bL}(\bx),
\end{equation}
and
\begin{equation}\label{prf8}
\prod_{m=1}^M \log (L_m + 1) \le \tsigma(\bzero).
\end{equation}
Using (\ref{prf8}) we have
\begin{equation}\label{prf9}
|Y|^2 \prod_{m=1}^M \log(L_m + 1) \le \tsigma(\bzero) |Y|^2 
		\le \sum_{\bell \in \lL - \lL} \tsigma(\bell) \biggl|\sum_{\ceta \in Y} e(\bell^TA\ceta)\biggr|^2.
\end{equation}
Then using (\ref{prf1}) and (\ref{prf7}) we find that
\begin{align}\label{prf10}
\begin{split}
\sum_{\bell \in \lL - \lL} \tsigma(\bell)\biggl|\sum_{\ceta \in Y} e(\bell^TA\ceta)\biggr|^2
	&= \sum_{\bxi \in X} v(\bxi) \sum_{\bell \in \lL - \lL} \tsigma(\bell) e(\bell^TA\bxi)\\
	&\le |Y| \sum_{\bxi \in X} \sigma(A\bxi)\\
	&\le |Y| \sum_{\bxi \in X} F_{\bL}(A\bxi)\\
	&= |Y| \prod_{m=1}^M L_m + |Y| \sum_{\substack{\bxi \in X\\\bxi \not= \bzero}} F_{\bL}(A\bxi). 
\end{split}
\end{align}
The inequality (\ref{intro10}) in the statement of Theorem \ref{thmintro1} follows from (\ref{prf9}) and (\ref{prf10}).

To verify Corollary \ref{corintro2} we apply Theorem \ref{thmintro1} with 
\begin{equation*}\label{prff13}
L_m \le |Y|^{\frac1M} < L_m + 1,\quad\text{for each $m = 1, 2, \dots , M$}.
\end{equation*}
Then (\ref{intro10}) implies that
\begin{equation}\label{prf14}
|Y|\biggl(\frac{\log |Y|}{M}\biggr)^M - |Y| \le \sum_{\substack{\bxi \in X\\\bxi \not= \bzero}} F_{\bL}(A\bxi) 
	\le \sum_{\substack{\bxi \in X\\\bxi \not= \bzero}} F(A\bxi).
\end{equation}

\section{Proof of Theorem \ref{thmintro3} and Corollary \ref{cormet4}}

Let $G_1$ be a compact abelian group, $\mu_1$ a Haar measure on the Borel subsets of $G_1$ normalized so that
$\mu_1\bigl(G_1\bigr) = 1$, and $\Gamma_1$ the dual group.  That is, $\Gamma_1$ is the group of continuous homomorphisms
$\gamma_1 : G_1 \rightarrow \T$, where 
\begin{equation*}\label{ub0.5}
\T = \{z\in \C: |z| = 1\}
\end{equation*}
is the circle group.  It follows from the Pontryagin duality theorem that $\Gamma_1$ is discrete.  Let $G_2$, $\mu_2$, and 
$\Gamma_2$, be another such triple.  

\begin{theorem}\label{thmub5}  Assume that $\p:G_1 \rightarrow G_2$ is a continuous, surjective, homomorphism.  Then 
for every function $f:G_2 \rightarrow \C$ in $L^1\bigl(G_2, \mu_2\bigr)$ we have
\begin{equation}\label{ub2}
\int_{G_1} f\bigl(\p(x)\bigr)\ \dmu_1(x) = \int_{G_2} f(y)\ \dmu_2(y).
\end{equation}
\end{theorem}

\begin{proof}  Let $\gamma_2$ be a nonprincipal character in the dual group $\Gamma_2$.  Then the composition
\begin{equation}\label{ub3}
\gamma_2\circ\p : G_1 \rightarrow \T
\end{equation}
is clearly a continuous homomorphism, and is therefore a character in $\Gamma_1$.  As $\gamma_2$ is not principal,
there exists a point $y_0$ in $G_2$ such that $\gamma_2(y_0) \not= 1$.  Because $\p$ is surjective, there exists a
point $x_0$ in $G_1$ such that $\p(x_0) = y_0$.  Then we have
\begin{equation*}\label{ub4}
\gamma_2\bigl(\p(x_0)\bigr) = \gamma_2(y_0) \not= 1,
\end{equation*}
and it follows that the composition (\ref{ub3}) is not the principle character in $\Gamma_1$.  From the orthogonality relations
for characters we find that
\begin{equation*}\label{ub5}
\int_{G_1} \gamma_2\bigl(\p(x)\bigr)\ \dmu_1(x) = \int_{G_2} \gamma_2(y)\ \dmu_2(y) = 0.
\end{equation*}
If $\gamma_2$ {\it is} the principal character in $\Gamma_2$, then it is obvious that the composition (\ref{ub3}) is
the principal character in $\Gamma_1$.  Hence in this case we get
\begin{equation*}\label{ub5.1}
\int_{G_1} \gamma_2\bigl(\p(x)\bigr)\ \dmu_1(x) = \int_{G_2} \gamma_2(y)\ \dmu_2(y) = 1.
\end{equation*}
Thus we have
\begin{equation}\label{ub6}
\int_{G_1} \gamma_2\bigl(\p(x)\bigr)\ \dmu_1(x) = \int_{G_2} \gamma_2(y)\ \dmu_2(y)
\end{equation}
for all characters $\gamma_2$ in $\Gamma_2$.   If $F_2 \subseteq \Gamma_2$ if a finite subset, and
\begin{equation*}\label{ub7}
T(y) = \sum_{\gamma_2 \in F_2} c(\gamma_2)\gamma_2(y)
\end{equation*}
is a finite linear combination of characters from $\Gamma_2$ with complex coefficients, then (\ref{ub6}) implies that
\begin{equation}\label{ub8}
\int_{G_1} T\bigl(\p(x)\bigr)\ \dmu_1(x) = \int_{G_2} T(y)\ \dmu_2(y).
\end{equation}
Because the set of all finite linear combinations of characters from $\Gamma_2$ is dense in $L^1\bigl(G_2, \mu_2\bigr)$, it
follows in a standard manner that (\ref{ub2}) holds.
\end{proof}

\begin{corollary}\label{corub6}  Let $E \subseteq G_2$ be a Borel set.  Then we have
\begin{equation}\label{ub9}
\mu_1\bigl(\p^{-1}(E)\bigr) = \mu_1\big\{x \in G_1 : \p(x) \in E\big\}  = \mu_2(E).
\end{equation}
\end{corollary}

\begin{proof}  This is (\ref{ub2}) in the special case $f(y) = \chi_E(y)$, where $\chi_E$ is the characteristic
function of the Borel set $E$.
\end{proof}

We now return to consideration of the groups 
\begin{equation}\label{ub20}
G_1 = (\R/\Z)^{MN},\quad\text{and}\quad G_2 = (\R/\Z)^M,
\end{equation}
specified in (\ref{intro2}).  We continue to write the elements of $G_1$ as $M\times N$ matrices with entries in the additive group
$\R/\Z$.  If $\bxi$ is a (column) vector in $\Z^N$ then 
\begin{equation}\label{ub22}
A \mapsto A\bxi = \biggl(\sum_{n=1}^N \alpha_{mn} \xi_n\biggr)
\end{equation}
defines a continuous homomorphism from $G_1$ into $G_2$.  If $\bxi \not= \bzero$ then it follows that (\ref{ub22})
is surjective, and therefore the conclusions of Theorem \ref{thmub5} and Corollary \ref{corub6} can be applied to 
this map.

\begin{lemma}\label{lemub7}  Let $G_2$ be as in {\rm (\ref{ub20})}.  If $0 < \delta \le 1$ then
\begin{equation}\label{ub26}
\mu_2\bigg\{\bbeta \in G_2 : \prod_{m=1}^M\bigl(2\|\beta_m\|\bigr) \le \delta\bigg\} 
	= \frac{1}{(M-1)!} \int_0^{\delta}\bigl(-\log x\bigr)^{M-1}\ \dx.
\end{equation}
\end{lemma}

\begin{proof}  Because $G_2 = (\R/\Z)^M$ is a product set, the set of coordinate functions
\begin{equation*}\label{ub27}
\bbeta \mapsto \log\bigl(2\|\beta_m\|\bigr),\quad\text{for each}\ m = 1, 2, \dots , M,
\end{equation*}
is a collection of $M$ independent, identically distributed, random variables on the probability space $\bigl(G_2, \mu_2\bigr)$.
The density function of each of these random variables is
\begin{equation}\label{ub28}
h(x) = \begin{cases}  e^x & \text{if $x \le 0$,}\\
                                          0 & \text{if $0 < x$.}\end{cases}
\end{equation}
That is, for each index $m$, $1 \le m \le M$, and $-\infty < u < v < \infty$, we have
\begin{equation*}\label{ub29}
\mu_2\big\{\bbeta \in G_2 : u < \log\bigl(2\|\beta_m\|\bigr) \le v\big\} = \int_u^v h(x)\ \dx.
\end{equation*}
Therefore the density function associated to the sum
\begin{equation*}\label{ub29.1}
\bbeta \mapsto \sum_{m=1}^M \log\bigl(2\|\beta_m\|\bigr)
\end{equation*}
of $M$ independent random variables, is the $M$-fold
convolution 
\begin{equation*}\label{ub29.2}
h*h* \cdots *h(x) = h^{(M)}(x).
\end{equation*}  
In order to compute this density, observe that the Fourier transform of $h$ is 
\begin{equation}\label{ub30}
\uh(y) = \int_{\infty}^{\infty} h(x) e(-xy)\ \dx = (1 - 2\pi iy)^{-1}.
\end{equation}
And therefore the Fourier transform of $h^{(M)}(x)$ is $(1 - 2\pi iy)^{-M}$.  By differentiating both sides of (\ref{ub30}) repeatedly with
respect to $y$, we obtain the identity
\begin{equation*}\label{ub31}
\frac{1}{(M-1)!} \int_{-\infty}^0 (-x)^{M-1} e^x e(-xy)\ \dx = (1- 2\pi iy)^{-M}.
\end{equation*}
It follows that 
\begin{equation*}\label{ub32}
h^{(M)}(x) =  \frac{(-x)^{M-1} e^x}{(M-1)!}\quad \text{if $x \le 0$},
\end{equation*}
and
\begin{equation*}\label{ub33}
h^{(M)}(x) = 0\quad \text{if $0 < x$}.
\end{equation*}
In particular, if $0 < \delta \le 1$ we get
\begin{align}\label{ub34}
\begin{split}
\mu_2\bigg\{\bbeta \in G_2 : \sum_{m=1}^M &\log\bigl(2\|\beta_m\|\bigr) \le \log \delta\bigg\}\\
	&= \frac{1}{(M-1)!} \int_{-\infty}^{\log \delta} (-x)^{M-1} e^x\ \dx\\
	&= \frac{1}{(M-1)!} \int_0^{\delta}\bigl(-\log x\bigr)^{M-1}\ \dx.
\end{split}
\end{align}
Then (\ref{ub34}) is equivalent to (\ref{ub26}).
\end{proof}

\begin{corollary}\label{corub8}  Let $G_1$ be as in {\rm (\ref{ub20})}, and let $\bxi$ be a nonzero lattice point in $\Z^N$.  
If $1 \le \lambda < \infty$ then
\begin{equation}\label{ub37}
\mu_1\big\{A \in G_1: \lambda \le F(A\bxi)\big\} = \frac{1}{(M-1)!} \int_{\lambda}^{\infty} \bigl(\log x\bigr)^{M-1} x^{-2}\ \dx.
\end{equation}
\end{corollary}

\begin{proof}  From Corollary \ref{corub6}  we have
\begin{equation}\label{ub38}
\mu_1\big\{A \in G_1: \lambda \le F(A\bxi)\big\} = \mu_2\{\bbeta \in G_2 : \lambda \le F(\bbeta)\}.
\end{equation}
The measure of the set on the right of (\ref{ub38}) follows from (\ref{ub26}) by a simple change of variables.
\end{proof}

We are now in position to bound the $\mu_1$-measure of the set
\begin{equation*}\label{ub39}
\bigg\{A \in G_1 : \lambda \le \sum_{\substack{\bxi \in X\\\bxi \not= \bzero}} F(A\bxi)\bigg\},
\end{equation*}
where $X \subseteq \Z^N$ is a finite, nonempty subset of lattice points. 

\begin{theorem}\label{thmub9}  Let $G_1$ be as in {\rm (\ref{ub20})}, and let $X \subseteq \Z^N$ be a finite, nonempty subset of
lattice points with cardinaltiy $|X|$.  If $1 \le \lambda < \infty$ then
\begin{equation}\label{ub50}
\mu_1\bigg\{A \in G_1 : \lambda \le \sum_{\substack{\bxi \in X\\\bxi \not= \bzero}} F(A\bxi)\bigg\} 
		\le \frac{|X|}{M!} \int_{\lambda}^{\infty} \bigl(\log x\bigr)^M x^{-2}\ \dx.
\end{equation}
\end{theorem}

\begin{proof}  We may assume that $X$ does not contain $\bzero$.  Let $1 \le \eta < \infty$, and for each lattice point $\bxi$ in $X$, define
\begin{equation*}\label{ub52}
D(\eta, \bxi) = \big\{A \in G_1: F(A\bxi) < \eta\big\}.
\end{equation*}
By Corollary \ref{corub8} we have
\begin{align}\label{ub53}
\begin{split}
\mu_1\bigl(G_1 \setminus D(\eta, \bxi)\bigr) &= \mu_1\big\{A \in G_1: \eta \le F(A\bxi)\big\}\\
		&= \frac{1}{(M-1)!} \int_{\eta}^{\infty} \bigl(\log x\bigr)^{M-1} x^{-2}\ \dx.
\end{split}
\end{align}
Next we write
\begin{equation*}\label{ub54}
\D(\eta) = \bigcap_{\bxi \in X} D(\eta, \bxi),
\end{equation*}
so that using (\ref{ub53}) we have
\begin{align}\label{ub55}
\begin{split}
\mu_1\bigl(G_1 \setminus \D(\eta)\bigr) &= \mu_1\biggl(\bigcup_{\bxi \in X} \bigl(G_1 \setminus D(\eta, \bxi)\bigr)\biggr)\\
				       &\le \sum_{\bxi \in X} \mu_1\bigl(G_1 \setminus D(\eta, \bxi)\bigr)\\
				       &= \frac{|X|}{(M-1)!} \int_{\eta}^{\infty} \bigl(\log x\bigr)^{M-1} x^{-2}\ \dx.
\end{split}        
\end{align}
Using a standard argument we get the estimate
\begin{align}\label{ub57}
\begin{split}
\mu_1\bigg\{A \in \D(\eta) : \lambda \le &\sum_{\bxi \in X} F(A\bxi)\bigg\}\\
		&\le \lambda^{-1} \int_{\D(\eta)} \biggl(\sum_{\bxi \in X} F(A\bxi)\biggr)\ \dmu_1(A)\\
		&\le \lambda^{-1} \sum_{\bxi \in X} \biggl(\int_{D(\eta, \bxi)} F(A\bxi)\ \dmu_1(A)\biggr).
\end{split}
\end{align}
From (\ref{ub2}) and Corollary \ref{corub8}, the integral on the right of (\ref{ub57}) is
\begin{align}\label{ub58}
\begin{split}
\int_{D(\eta, \bxi)} F(A\bxi)\ \dmu_1(A) &= \int_{G_1} \chi_{D(\eta, \bxi)}\bigl(F(A\bxi)\bigr) F(A\bxi)\ \dmu_1(A)\\
			&= \int_{G_2} \chi_{D(\eta, \bxi)}\bigl(F(\bbeta)\bigr) F(\bbeta)\ \dmu_2(\bbeta)\\
			&= \frac{1}{(M-1)!} \int_1^{\eta} \bigl(\log x\bigr)^{M-1} x^{-1}\ \dx.
\end{split}
\end{align}
We combine (\ref{ub57}) and (\ref{ub58}) to obtain the inequality
\begin{align}\label{ub59}
\begin{split}
\mu_1\bigg\{A \in \D(\eta) : \lambda \le &\sum_{\bxi \in X} F(A\bxi)\bigg\}\\
			&\le \frac{|X|}{\lambda (M-1)!}  \int_1^{\eta} \bigl(\log x\bigr)^{M-1} x^{-1}\ \dx.
\end{split}
\end{align}
And then we combine (\ref{ub55}) and (\ref{ub59}) to get
\begin{align}\label{ub60}
\begin{split}
\mu_1\bigg\{&A \in G_1 : \lambda \le \sum_{\bxi \in X} F(A\bxi)\bigg\}\\
	&\le \mu_1\bigg\{A \in \D(\eta) : \lambda \le \sum_{\bxi \in X} F(A\bxi)\bigg\} + \mu_1\bigl(G_1 \setminus \D(\eta)\bigr)\\
	&\le \frac{|X|}{(M-1)!}\biggl(\lambda^{-1} \int_1^{\eta} \bigl(\log x\bigr)^{M-1} x^{-1}\ \dx 
		+  \int_{\eta}^{\infty} \bigl(\log x\bigr)^{M-1} x^{-2}\ \dx\biggr)	
\end{split}
\end{align}
The parameter $\eta$ in the upper bound on the right of (\ref{ub60}) is at our disposal.  A simple calculation shows
that the right hand side of (\ref{ub60}) is minimized at $\eta = \lambda$.  We find that
\begin{align}\label{ub61}
\begin{split}
\lambda^{-1} \int_1^{\lambda} \bigl(\log x\bigr)^{M-1}& x^{-1}\ \dx 
		+  \int_{\lambda}^{\infty} \bigl(\log x\bigr)^{M-1} x^{-2}\ \dx\\
		&= \frac{1}{M} \int_{\lambda}^{\infty} \bigl(\log x\bigr)^M x^{-2}\ \dx,
\end{split}
\end{align}
and the theorem is proved.
\end{proof}

In applications of Theorem \ref{thmub9} the identity
\begin{equation}\label{ub62}
\frac{1}{M!} \int_{\lambda}^{\infty} \bigl(\log x\bigr)^M x^{-2}\ \dx = \lambda^{-1} \sum_{m=0}^M \frac{(\log \lambda)^m}{m!}
\end{equation}
is useful.  We also note that if $1 \le \lambda_1 < \infty$ and $1 \le \lambda_2 < \infty$, then
\begin{align}\label{ub63}
\begin{split}
\sum_{m=0}^M \frac{(\log \lambda_1 + \log \lambda_2)^m}{m!}
	&= \sum_{m=0}^M \sum_{l=0}^m \frac{(\log \lambda_1)^l}{l!} \frac{(\log \lambda_2)^{m-l}}{(m-l)!}\\
	&= \sum_{l=0}^M  \frac{(\log \lambda_1)^l}{l!} \sum_{m=l}^M \frac{(\log \lambda_2)^{m-l}}{(m-l)!}\\
	&\le \biggl(\sum_{l=0}^M \frac{(\log \lambda_1)^l}{l!}\biggr)\biggl( \sum_{k=0}^M \frac{(\log \lambda_2)^k}{k!}\biggr).
\end{split}
\end{align}
For example, if $0 < \delta < 1$ then there exists a unique real number $\lambda$ such that $1 < \lambda < \infty$ and
\begin{equation*}\label{ub64}
\delta = \frac{1}{M!} \int_{\lambda}^{\infty} \bigl(\log x\bigr)^M x^{-2}\ \dx.
\end{equation*}
Using (\ref{ub62}) we find that
\begin{align*}\label{ub65}
\begin{split}
\delta &< \lambda^{-1} \sum_{m=0}^M 2^{M-m} \frac{(\log \lambda)^m}{m!}\\
               &= 2^M \lambda^{-1} \sum_{m=0}^M \frac{\bigl(\hh \log \lambda\bigr)^m}{m!}\\
               &\le 2^M \lambda^{-\h},
\end{split}
\end{align*}
and therefore
\begin{equation}\label{ub66}
\lambda < 4^M \delta^{-2}.
\end{equation}
Then using (\ref{ub62}) and (\ref{ub66}), we get
\begin{align}\label{ub67}
\begin{split}
\lambda &= \delta^{-1} \sum_{m=0}^M \frac{(\log \lambda)^m}{m!}\\
	      &\le \delta^{-1} \sum_{m=0}^M \frac{\bigl(\log 4^M + \log \delta^{-2}\bigr)^m}{m!}\\
	      &\le 4^M \delta^{-1} \sum_{m=0}^M \frac{\bigl(2 \log \delta^{-1}\bigr)^m}{m!}\\
	      &\le 8^M \delta^{-1} \sum_{m=0}^M \frac{\bigl(\log \delta^{-1}\bigr)^m}{m!}.
\end{split}
\end{align}

We now prove Theorem \ref{thmintro3}.  Let $\lambda$ be selected so that
\begin{equation*}\label{ub70}
\ep |X|^{-1} = \frac{1}{M!} \int_{\lambda}^{\infty} \bigl(\log x\bigr)^M x^{-2}\ \dx.
\end{equation*}
Then Theorem \ref{thmub9} implies that 
\begin{equation}\label{ub71}
\mu_1\bigg\{A \in G_1 : \lambda \le \sum_{\bxi \in X} F(A\bxi)\bigg\} \le \ep.
\end{equation}
We apply the inequality (\ref{ub67}) with $\delta = \ep |X|^{-1}$.  It follows from (\ref{ub71}) that the inequality (\ref{intro18})
holds at all point $A$ in $G_1$ outside a set of $\mu_1$-measure at most $\ep$.

Next we prove Corollary \ref{cormet4}.  It follows from Corollary \ref{corub6} that 
\begin{equation*}\label{ub78}
\mu_1\big\{A \in G_1 : F(A\bxi) < \infty\big\} = 1
\end{equation*}
for each point $\bxi \not= \bzero$ in $\Z^N$.  Therefore almost all points $A$ in $G_1$ belong to the subset
\begin{equation}\label{ub79}
\Y = \bigcap_{\substack{\bxi \in \Z^N\\ \bxi \not= \bzero}} \big\{A \in G_1 : F(A\bxi) < \infty\big\}.
\end{equation}

For each $\ell = 1, 2, \dots ,$ we define
\begin{equation}\label{ub80}
\ep_{\ell}^{-1} = \log 3|Y_{\ell}| \bigl(\log \log 27|Y_{\ell}|\bigr)^{1 + \eta},
\end{equation}
and
\begin{equation*}\label{ub81}
\A_{\ell} = \bigg\{A \in G_1: 8^M \ep_{\ell}^{-1} |Y_{\ell}| \sum_{m=0}^M \frac{\bigl(\log \ep_{\ell}^{-1}|Y_{\ell}|\bigr)^m}{m!}              		        
                             < \sum_{\substack{\bxi \in Y_{\ell}\\\bxi \not= \bzero}} F(A\bxi)\bigg\}.
\end{equation*}
Then using (\ref{intro18}) and (\ref{met1}) we find that
\begin{equation*}\label{ub82}
\sum_{\ell = 1}^{\infty} \mu_1\bigl(\A_{\ell}\bigr) \le \sum_{\ell = 1}^{\infty} \ep_{\ell} < \infty.
\end{equation*}
Thus by the Borel-Cantelli lemma almost all points $A$ in $G_1$ belong to the subset
\begin{equation*}\label{ub82.5}
\zz = \bigcap_{L=1}^{\infty} \bigcup_{\ell = L}^{\infty} \bigl(G_1 \setminus \A_{\ell}\big).
\end{equation*}

Now suppose that $A$ belongs to $\Y\cap \zz$.  Then $A$ belongs to only finitely many of the subsets
$\A_{\ell}$.   That is, there exists a positive integer $L = L(A)$ such that
\begin{equation}\label{ub83} 
 \sum_{\substack{\bxi \in Y_{\ell}\\\bxi \not= \bzero}} F(A\bxi) 
		\le 8^M \ep_{\ell}^{-1} |Y_{\ell}| \sum_{m=0}^M \frac{\bigl(\log \ep_{\ell}^{-1}|Y_{\ell}|\bigr)^m}{m!}
\end{equation}
for $L \le \ell$.  Because $A$ also belongs to $\Y$, the numbers
\begin{equation}\label{ub84}
 \sum_{\substack{\bxi \in Y_{\ell}\\\bxi \not= \bzero}} F(A\bxi),\quad\text{where $\ell = 1, 2, \dots , L(A),$}
\end{equation}
are finite.  From (\ref{ub83}) and (\ref{ub84}) we conclude that
\begin{equation*}\label{ub84.5} 
 \sum_{\substack{\bxi \in Y_{\ell}\\\bxi \not= \bzero}} F(A\bxi) 
		\ll_{A, M} \ep_{\ell}^{-1} |Y_{\ell}| \bigl(\log \ep_{\ell}^{-1}|Y_{\ell}|\bigr)^M
\end{equation*}
for {\it all} positive integers $\ell = 1, 2, \dots $.  Then using (\ref{ub80}) we find that
\begin{equation}\label{ub85} 
 \sum_{\substack{\bxi \in Y_{\ell}\\\bxi \not= \bzero}} F(A\bxi) 
		\ll_{A, \eta, M} |Y_{\ell}| \bigl(\log 3|Y_{\ell}|\bigr)^{M+1} \bigl(\log\log 27|Y_{\ell}|\bigr)^{1 + \eta}
\end{equation}
for all positive integers $\ell$.  Finally, if $X$ is in the collection of subsets $\X$, then $X \subseteq Y_{\ell}$
for some positive integer $\ell$, and therefore
\begin{equation}\label{ub86} 
 \sum_{\substack{\bxi \in X\\\bxi \not= \bzero}} F(A\bxi) 
		\ll_{A, \eta, M} |Y_{\ell}| \bigl(\log 3|Y_{\ell}|\bigr)^{M+1} \bigl(\log\log 27|Y_{\ell}|\bigr)^{1 + \eta}.
\end{equation}
Because $|Y_{\ell}| \le C_2 |X|$ for an absolute constant $C_2$, the bound (\ref{met3}) follows easily from (\ref{ub86}).

\section{Proof of Theorem \ref{thmsba1}}

We assume throughout this section that $A$ is an $M\times N$ matrix in $G_1 = (\R/\Z)^{MN}$, and $A$ is strongly 
badly approximable.  Using the notation introduced in (\ref{ls1}) and (\ref{ls2}), we conclude that
\begin{equation}\label{uba1}
\gamma(A) = \inf\big\{P_M\bigl(A\bxi\bigr) Q_N(\bxi) : \bxi \in \Z^N,\ \text{and}\ \bxi \not= \bzero\big\}
\end{equation}
is positive.  In particular we have
\begin{equation}\label{uba2}
0 < \gamma(A) \le \gamma(A, \K)
\end{equation}
for each subset $\K \subseteq \Z^N$ defined by (\ref{ls3}) and $\gamma(A, \K)$ defined by (\ref{ls7}).

Let
\begin{equation*}\label{uba3}
\D = \{\bd \in \Z^M : 1 \le d_m\ \text{for each}\ m = 1, 2, \dots , M\}.
\end{equation*}
For each point $\bd$ in $\D$ we write
\begin{equation*}\label{uba4}
|\bd| = d_1 + d_2 + \cdots + d_M,
\end{equation*}
and we define
\begin{equation*}\label{uba5}
B(\bd) = \big\{\bx \in (\R/\Z)^M : 2^{-d_m - 1} < \|x_m\| \le 2^{-d_m}\ \text{for each}\ m = 1, 2, \dots , M\}.
\end{equation*}
As each subset $B(\bd)$ contains $2^M$ connected subsets of equal measure, it follows that
\begin{equation}\label{uba6}
\mu_2\bigl(B(\bd)\bigr) = 2^M \prod_{m=1}^M (2^{-d_m} - 2^{-d_m - 1}) = 2^{-|\bd|}.
\end{equation}
If $\bx$ belongs to $B(\bd)$ we find that
\begin{equation}\label{uba7}
2^{|\bd| - M} \le F(\bx) = \prod_{m = 1}^M (2 \|x_m\|)^{-1} <  2^{|\bd|}.
\end{equation}
If $\bd$ and $\be$ are distinct elements of $\D$, then it is clear that $B(\bd)$ and $B(\be)$ are disjoint subsets.
Moreover, we have
\begin{equation*}\label{uba8}
\bigcup_{\bd \in \D} B(\bd) = \big\{\bx \in (\R/\Z)^M : 0 < \|x_m\| \ \text{for each}\ m = 1, 2, \dots , M\big\}.
\end{equation*}
Because the co-ordinates of each point $\bd$ in $\D$ are positive integers, it is clear that $M \le |\bd|$.  And
if $R$ is an integer such that $M \le R$, then
\begin{equation}\label{uba9}
\sum_{\substack{\bd \in \D\\|\bd| \le R}} 1 = \sum_{m = M}^R \sum_{\substack{\bd \in \D\\|\bd| = m}} 1
	= \sum_{m = M}^R \binom{m-1}{m-M} = \binom{R}{M}.
\end{equation}

Let $\bx \mapsto \Phi\bigl(\bx, B(\bd)\bigr)$ denote the characteristic function of the subset $B(\bd)$.  The set
$B(\bd)$ is contained in the union of $2^M$ subsets $\Delta(\bwy)$ of the form (\ref{ls100}).  In particular, we have
\begin{equation}\label{uba10}
B(\bd) \subseteq \bigcup \big\{\bx \in (\R/\Z)^M : \|x_m \pm (3)2^{-d_m-2}\| \le 2^{-d_m - 2}\big\},
\end{equation}
where the union on the right of (\ref{uba10}) is over the set of all $2^M$ choices of $\pm$ signs.  We apply Lemma \ref{lemls3}
to each subset on the right of (\ref{uba10}).  Then the inequality (\ref{ls101}) in the statement of Lemma \ref{lemls3} 
and (\ref{uba2}), imply that
\begin{equation}\label{uba11}
\sum_{\bxi \in \K} \Phi\bigl(A\bxi, B(\bd)\bigr) \le 4^M \gamma(A)^{-1} \mu_2\bigl(B(\bd)\bigr) |\K| + 12^M
\end{equation}
for each point $\bd$ in $\D$.
\begin{equation*}\label{uba12}
\end{equation*}

Now let $R$ be the unique positive integer such that
\begin{equation}\label{uba13}
2^{R - 1} \le 2^{-M} \gamma(A)^{-1} |\K| < 2^R.
\end{equation}
and therefore,
\begin{equation}\label{uba14}
R \le \frac{-\log \gamma(A) + \log |\K|}{\log 2} \ll_A \log 2|\K|.
\end{equation}
From (\ref{uba1}) we conclude that if $\bxi \not= \bzero$ belongs to $\K$, then
\begin{equation*}\label{uba15}
F(A\bxi) \le 2^{-M} \gamma(A)^{-1} Q_N(\bxi) \le 2^{-M} \gamma(A)^{-1} |\K| < 2^R.
\end{equation*}
It follows that for each point $\bxi \not= \bzero$ in $\K$ there exists a unique point $\bd$ in $\D$ such that $|\bd| \le R$ and
$A\bxi$ belongs to $B(\bd)$.  Then using (\ref{uba7}), (\ref{uba11}), we obtain the inequality
\begin{align}\label{uba16}
\begin{split}
\sum_{\substack{\bxi \in \K\\\bxi \not= \bzero}} F(A\bxi) 
	&= \sum_{\substack{\bd \in \D\\|\bd| \le R}} \sum_{\substack{\bxi \in \K\\\bxi \not= \bzero}} \Phi\bigl(A\bxi, B(\bd)\bigr) F(A\bxi)\\
	&\le \sum_{\substack{\bd \in \D\\|\bd| \le R}} 2^{|\bd|} \sum_{\substack{\bxi \in \K\\\bxi \not= \bzero}} \Phi\bigl(A\bxi, B(\bd)\bigr)\\
	&\le 4^M \gamma(A)^{-1} |\K| \sum_{\substack{\bd \in \D\\|\bd| \le R}} 2^{|\bd|} \mu_2\bigl(B(\bd)\bigr)
		+ 12^M \sum_{\substack{\bd \in \D\\|\bd| \le R}} 2^{|\bd|}\\
\end{split}
\end{align}
From (\ref{uba6}), (\ref{uba9}), and (\ref{uba14}), we find that
\begin{equation}\label{uba17}
\sum_{\substack{\bd \in \D\\|\bd| \le R}} 2^{|\bd|} \mu_2\bigl(B(\bd)\bigr) 
		= \sum_{\substack{\bd \in \D\\|\bd| \le R}} 1 = \binom{R}{M} \ll_A \bigl(\log 2|\K|\bigr)^M.
\end{equation}
In a similar manner using (\ref{uba13}), we get
\begin{equation}\label{uba18}
\sum_{\substack{\bd \in \D\\|\bd| \le R}} 2^{|\bd|} \le 2^R \sum_{\substack{\bd \in \D\\|\bd| \le R}} 1\ll_A |\K| \bigl(\log 2|\K|\bigr)^M.
\end{equation}
The inequality (\ref{sba16}) in the statement of Theorem \ref{thmsba1} follows from (\ref{uba16}), (\ref{uba17}), and (\ref{uba18}).

\section{Proof of Theorem \ref{thmsba2}}

Let $A = (\alpha_{mn})$ be an $M\times N$ real matrix with $L = M + N$.  We write $B$ for the $L \times L$ real matrix 
partitioned into blocks as
\begin{equation*}\label{trans10}
B = \begin{pmatrix}  \bone_M & A \\
                                  \bzero      & \bone_N \end{pmatrix},
\end{equation*}
where $\bone_M$ and $\bone_N$ are $M\times M$ and $N\times N$ identity matrices, respectively.  Then we write 
$\Delta = [\delta_{\ell}]$ for an $L \times L$ diagonal matrix with positive diagonal entries $\delta_{\ell}$ such that
\begin{equation}\label{trans11}
\det \Delta = \prod_{\ell = 1}^L \delta_{\ell} = 1.
\end{equation}
Using $B$ and $\Delta$ we define a lattice $\M \subseteq \R^L$ by
\begin{equation}\label{trans12}
\M = \big\{\Delta B \bm : \bm \in \Z^L\big\}.
\end{equation}
And we define the associated convex body $C_L \subseteq \R^L$ by
\begin{equation*}\label{trans13}
C_L = \big\{\bx \in \R^L : |\bx|_{\infty} \le 1\big\},
\end{equation*}
where
\begin{equation*}
|\bx|_{\infty} = \max\big\{|x_1|, |x_2|, \dots , |x_L|\big\}.
\end{equation*}
Let
\begin{equation}\label{trans14}
0 < \lambda_1 \le \lambda_2 \le \cdots \le \lambda_L < \infty
\end{equation}
be the successive minima of the lattice $\M$ with respect to the convex body $C_L$.  We have
\begin{equation}\label{trans15}
\det(\M) = \bigl|\det(\Delta B)\bigr| = 1,\quad\text{and}\quad \Vol_L\bigl(C_L\bigr) = 2^L.
\end{equation}
Therefore it follows from (\ref{trans15}) and Minkowski's inequality (see \cite[Chapter VIII, Theorem V]{cassels1971}) that
\begin{equation}\label{trans16}
(L!)^{-1} \le \lambda_1 \lambda_2 \cdots \lambda_L \le 1.
\end{equation}

The dual (or polar) lattice $\M^* \subseteq \R^L$ is given by
\begin{equation*}\label{trans19}
\M^* = \big\{\Delta^{-1} B^{-T} \bn : \bn \in \Z^L\big\},
\end{equation*}
where
\begin{equation*}\label{trans20}
B^{-T} = \begin{pmatrix}  \bone_M & \bzero \\
                                                  -A^T   & \bone_N \end{pmatrix}.
\end{equation*}
And the dual (or polar) convex body $C_L^*$ is
\begin{equation*}\label{trans21}
C_L^* = \big\{\bx \in \R^L : |\bx|_1 \le 1\big\},
\end{equation*}
where
\begin{equation*}
|\bx|_1 = |x_1| + |x_2| + \cdots + |x_L|.
\end{equation*}
Let
\begin{equation}\label{trans28}
0 < \lambda_1^* \le \lambda_2^* \le \cdots \le \lambda_L^* < \infty
\end{equation}
be the successive minima associated to the lattice $\M^*$ and the convex body $C_L^*$.  In this case we find that
\begin{equation*}\label{trans29}
\det(\M^*) = \bigl|\det(\Delta^{-1} B^{-T})\bigr| = 1,\quad\text{and}\quad \Vol_L\bigl(C_L^*\bigr) = \frac{2^L}{L!}.
\end{equation*}
Thus Minkowski's inequality for the dual successive minima is
\begin{equation}\label{trans30}
1 \le \lambda_1^* \lambda_2^* \cdots \lambda_L^* \le L!.
\end{equation}

The two sets of successive minima (\ref{trans14}) and (\ref{trans28}) are linked by an inequality of 
Mahler \cite{mahler1939} (see also \cite[Chapter VIII, Theorem VI]{cassels1971}), which asserts that
\begin{equation*}\label{trans41}
1 \le \lambda_{\ell} \lambda_{L - \ell +1}^* \le L!
\end{equation*}
for each integer $\ell = 1, 2, \dots , L$.  In particular, if $\ell = 1$ we have
\begin{equation}\label{trans42}
1 \le \lambda_1 \lambda_L^*.
\end{equation}
Then using (\ref{trans28}), (\ref{trans30}), and (\ref{trans42}), we get
\begin{equation}\label{trans43}
\frac{\bigl(\lambda_1^*\bigr)^{L-1}}{L!} \le \frac{\lambda_1^* \lambda_2^* \cdots \lambda_{L-1}^*}{L!}
		\le \frac{\lambda_1\bigl(\lambda_1^* \lambda_2^* \cdots \lambda_L^*\bigr)}{L!} \le \lambda_1.
\end{equation}

Now assume that $A^T$ is strongly badly approximable.  We have already observed that each submatrix $A(I, J)^T$ is strongly 
badly approximable.  In particular, each column  of $A^T$ is strongly badly approximable, and therefore each column 
of $A^T$ is badly approximable.  By the basic transference principle \cite[Chapter V, Corollary to Theorem II]{Cassels1965},
each row of the matrix $A$ is badly approximable.  Then (\ref{sba7}) implies that
\begin{equation}\label{trans44}
0 < \big\|\sum_{n=1}^N \alpha_{mn} \xi_n \Big\|
\end{equation}
for each point $\bxi \not= \bzero$ in $\Z^N$.

Because $A^T$ is strongly badly approximable, there exists a positive constant 
$\gamma\bigl(A^T\bigr)$ such that for each vector $\bu \not= \bzero$ in $\Z^M$, we have
\begin{equation}\label{sba24}
\gamma\bigl(A^T\bigr) \le \biggl(\prod_{m=1}^M \bigl(|u_m| + 1\bigr)\biggr)
		\biggl(\prod_{n=1}^N \Big\|\sum_{m=1}^M u_m \alpha_{mn}\Big\|\biggr).
\end{equation}
We will show that there exists a positive constant $\gamma(A)$ such that for each vector $\bxi \not= \bzero$ in $\Z^N$, 
we have
\begin{equation}\label{sba25}
\gamma(A) \le \biggl(\prod_{m=1}^M \Big\|\sum_{n=1}^N \alpha_{mn} \xi_n\Big\|\biggr) 
                    \biggl(\prod_{n=1}^N\bigl(|\xi_n| + 1\bigr)\biggr).
\end{equation}
Our proof of (\ref{sba25}) will be by induction on the positive integer $M$.

Let $\bxi \not= \bzero$ be a point in $\Z^N$ and let $\ceta$ be a point in $\Z^M$ such that
\begin{equation}\label{trans46}
\biggl|\eta_m + \sum_{n=1}^N \alpha_{mn} \xi_n\biggr| = \Big\|\sum_{n=1}^N \alpha_{mn} \xi_n\Big\|
\end{equation}
for each $m = 1, 2, \dots , M$.  It will be convenient to define the vector $\bpsi$ in $\Z^L$ by
\begin{equation*}\label{trans47}
\bpsi = \begin{pmatrix} \ceta \\
                                          \bxi\end{pmatrix},
\end{equation*}
where $\ceta$ belongs to $\Z^M$ and $\bxi \not= \bzero$ belongs to $\Z^N$.  In view of (\ref{trans44}), we define $R$ to 
be the unique positive real number such that
\begin{equation}\label{trans48}
R^L =  \biggl(\prod_{m=1}^M \Big\|\sum_{n=1}^N \alpha_{mn} \xi_n\Big\|\biggr)\biggl(\prod_{n=1}^N \bigl(|\xi_n| + 1\bigr)\biggr).
\end{equation}
Next we select the $L\times L$ diagonal matrix $\Delta = [\delta_{\ell}]$ so that
\begin{equation}\label{trans49}
R = \delta_m \Big\|\sum_{n=1}^N \alpha_{mn} \xi_n\Big\|\quad \text{if $1 \le m \le M$,}
\end{equation}
and
\begin{equation}\label{trans50}
R = \delta_{M+n} \bigl(|\xi_n| + 1\bigr) \quad \text{if $1 \le n \le N$.}
\end{equation}
It follows from (\ref{trans48}) that $\Delta$ satisfies the condition (\ref{trans11}).  
As $\Delta B \bpsi$ is a nonzero point in the lattice $\M$ defined by (\ref{trans12}), we have
\begin{equation}\label{trans51}
\lambda_1 \le |\Delta B \bpsi|_{\infty}.
\end{equation}
Then using (\ref{trans49}) and (\ref{trans50}), we find that (\ref{trans51}) can be written as
\begin{equation}\label{trans52}
\lambda_1 \le |\Delta B \bpsi|_{\infty} = R.
\end{equation}

It is well known (see \cite[Chapter VIII, Lemma 1]{cassels1971}) that there exists a point  $\bw \not= \bzero$ in $\Z^L$ such that
\begin{equation}\label{trans71}
\bigl|\Delta^{-1} B^{-T} \bw\bigr|_1 = \lambda_1^*.
\end{equation}
Again it will be convenient to partition the column vector $\bw$ as
\begin{equation*}\label{trans74}
\bw = \begin{pmatrix} \bu\\
                                       \bv\end{pmatrix},
\end{equation*}
where $\bu$ is a point in $\Z^M$ and $\bv$ is a point in $\Z^N$.  Then (\ref{trans71}) can be written as
\begin{align}\label{trans78}
\begin{split}
R^{-1} \sum_{m=1}^M |u_m| \biggl|\eta_m + &\sum_{n=1}^N \alpha_{mn} \xi_n\biggr|\\
	&+  R^{-1} \sum_{n=1}^N \bigl(|\xi_n| + 1\bigr) \biggl|v_n - \sum_{m=1}^M u_m \alpha_{mn}\biggr| = \lambda_1^*
\end{split}	
\end{align}
If $\bu = \bzero$ then $\bv \not= \bzero$, and (\ref{trans78}) leads to the inequality
\begin{equation}\label{trans82}
R^{-1} \le R^{-1} \sum_{n=1}^N\bigl(|\xi_n| + 1\bigr)|v_n| = \lambda_1^*.
\end{equation}
Then (\ref{trans43}), (\ref{trans52}), and (\ref{trans82}), imply that
\begin{equation}\label{trans83}
\frac{1}{L!} \le R^L.
\end{equation}
Clearly (\ref{trans43}), (\ref{trans52}), and (\ref{trans83}), show that $R^L$ is bounded from below by a positive constant
that depends on $L$, but not on the point $\bxi \not= \bzero$ in $\Z^N$.

For the remainder of the proof we assume that $\bu \not= \bzero$.  In this case it is clear from 
the definition of $\lambda_1^*$ that we must have
\begin{equation}\label{trans85}
\biggl|v_n - \sum_{m=1}^M u_m \alpha_{mn}\biggr| = \Big\|\sum_{m=1}^M u_m \alpha_{mn}\Big\|
\end{equation}
for each integer $n = 1, 2, \dots, N$.  Therefore (\ref{trans46}) and (\ref{trans85}) imply that (\ref{trans78}) can be written as
\begin{equation}\label{trans86}
\sum_{m=1}^M \delta_m^{-1} |u_m| 
	                          +  \sum_{n=1}^N \delta_{M+n}^{-1}\Big\|\sum_{m=1}^M u_m \alpha_{mn}\Big\| = \lambda_1^*.
\end{equation}
Then using (\ref{trans11}) and (\ref{sba24}) we have
\begin{align}\label{trans90}
\begin{split}
\gamma\bigl(A^T\bigr) &\le \biggl(\prod_{m=1}^M \bigl(|u_m| + 1\bigr)\biggr)
		\biggl(\prod_{n=1}^N \Big\|\sum_{m=1}^M u_m \alpha_{mn}\Big\|\biggr)\\
         &= \biggl(\prod_{m=1}^M \delta_m^{-1} \bigl(|u_m| + 1\bigr)\biggr)
  				\biggl(\prod_{n=1}^N \delta_{M+n}^{-1} \Big\|\sum_{m=1}^M u_m\alpha_{mn}\Big\| \biggr).
\end{split}
\end{align}

We now argue by induction on $M$.  If $M = 1$ then $u_1 \not= 0$.  We use the identity (\ref{trans86}), and we
apply the arithmetic-geometric mean inequality to the right hand side of (\ref{trans90}).  In this way we obtain the inequality				
\begin{align}\label{trans91}
\begin{split}				
L \gamma\bigl(A^T\bigr)^{\frac{1}{L}} &\le \delta_1^{-1} \bigl(|u_1| + 1\bigr) 
							  + \sum_{n=1}^N \delta_{M+n}^{-1} \Big\|\sum_{m=1}^M u_m \alpha_{mn}\Big\|\\
		&\le 2 \delta_1^{-1} |u_1| + 2 \sum_{n=1}^N \delta_{M+n}^{-1} \Big\|\sum_{m=1}^M u_m \alpha_{mn}\Big\|\\
		&= 2 \lambda_1^*.
\end{split}
\end{align}	
Then (\ref{trans43}), (\ref{trans52}), and (\ref{trans91}), imply that $R^L$ is bounded from below by a positive constant that
does not depend on the point $\bxi \not= \bzero$ in $\Z^N$.		

Finally, we assume that $2 \le M$.  We have already remarked that each submatrix $A(I, J)^T$ is strongly badly approximable,
and in particular the submatrix of $A^T$ obtained by removing the column indexed by $k$, where $1 \le k \le M$, is
strongly badly approximable.  Hence by the inductive hypothesis the submatrix of $A$ obtained by removing the row
indexed by $k$ is strongly badly approximable.  Thus there exists a positive constant $\gamma_k = \gamma_k(A)$ such
that the inequality
\begin{equation}\label{trans94}
0 < \gamma_k(A) \le \biggl(\prod_{\substack{m=1\\m \not= k}}^M\Big\|\sum_{n=1}^N \alpha_{mn} \xi_n\Big\|\biggr)
		\biggl(\prod_{n=1}^N \bigl(|\xi_n| + 1\bigr)\biggr)
\end{equation}
holds for all $\bxi \not= \bzero$ in $\Z^N$.  Using (\ref{trans94}) we obtain the inequality
\begin{align}\label{trans96}
\begin{split}
\gamma_k(A) \Big\|\sum_{n=1}^N \alpha_{kn} \xi_n \Big\| \le \biggl(\prod_{m=1}^M \Big\|\sum_{n=1}^N \alpha_{mn} \xi_n\Big\|\biggr)
	\biggl(\prod_{n=1}^N \bigl(|\xi_n| + 1\bigr)\biggr) = R^L
\end{split}
\end{align}
for each integer $k = 1, 2, \dots , M$.  Again we use the identity (\ref{trans86}), and we apply the arithmetic-geometric mean 
inequality to the right hand side of (\ref{trans90}).  This leads to the estimate
\begin{align}\label{trans97}
\begin{split}				
L \gamma\bigl(A^T\bigr)^{\frac{1}{L}} &\le \sum_{m=1}^M \delta_m^{-1} \bigl(|u_m| + 1\bigr) 
							  + \sum_{n=1}^N \delta_{M+n}^{-1} \Big\|\sum_{m=1}^M u_m \alpha_{mn}\Big\|\\				
	&= \lambda_1^* +  \sum_{m=1}^M \delta_m^{-1}\\
	&= \lambda_1^* + R^{-1} \sum_{m=1}^M \Big\|\sum_{n=1}^N \alpha_{mn} \xi_n \Big\|.
\end{split}
\end{align}
We apply the inequalities (\ref{trans43}), (\ref{trans52}), and (\ref{trans96}) to the right hand side of
(\ref{trans97}).  In this way we arrive at the bound
\begin{align}\label{trans93}
\begin{split}
L \gamma\bigl(A^T\bigr)^{\frac{1}{L}} &\le \lambda_1^* + R^{L-1} \sum_{m=1}^M \gamma_m(A)^{-1}\\
	&\le \bigl(L!\bigr)^{\frac{1}{L-1}} R^{\frac{1}{L-1}} + R^{L-1} \sum_{m=1}^M \gamma_m(A)^{-1}.
\end{split}
\end{align} 
The inequality (\ref{trans93}) shows that $R^L$ is bounded from below by a positive constant that does not 
depend on the point $\bxi \not= \bzero$ in $\Z^N$.  This verifies (\ref{sba25}), and completes the proof of Theorem \ref{thmsba2}.

\end{document}